\documentclass{article}
\usepackage[utf8]{inputenc}
\usepackage{amsmath, amssymb, amsthm}
\newtheorem{theorem}{Theorem}
\newtheorem{lemma}[theorem]{Lemma}

\newtheorem{definition}[theorem]{Definition}
\newtheorem{remark}[theorem]{Remark}
\usepackage[normalem]{ulem}
\DeclareMathOperator*{\argmin}{arg\,min}
\usepackage[ruled,vlined,linesnumbered]{algorithm2e}
\usepackage{csquotes}
\usepackage{subcaption}
\usepackage{tikz}
\usetikzlibrary{positioning}
\tikzset{
    position/.style args={#1:#2 from #3}{
        at=(#3), anchor=#1+180, shift=(#1:#2)
    }
}
\usetikzlibrary{decorations.pathreplacing}

\definecolor{myLightGray}{RGB}{191,191,191}
\definecolor{myGray}{RGB}{160,160,160}
\definecolor{myDarkGray}{RGB}{144,144,144}
\definecolor{myDarkRed}{RGB}{167,114,115}
\definecolor{myRed}{RGB}{250,00,00}
\definecolor{myGreen}{RGB}{00,250,00}

\usepackage{placeins}  
\usepackage{lineno}

\usepackage{xcolor}
\usepackage[english]{babel}
\usepackage[textsize=footnotesize]{todonotes}

\usepackage{hyperref}
\hypersetup{
    colorlinks,
    linkcolor={red!50!black},
    citecolor={blue!50!black},
    urlcolor={blue!80!black}
}
\usepackage{cleveref}

\usepackage{zref-savepos}
\usepackage{pict2e}
\usepackage[
    style=alphabetic,
  ]{biblatex}
\newcounter{NoTableEntry}
\renewcommand*{\theNoTableEntry}{NTE-\the\value{NoTableEntry}}

\usepackage{mathtools}                  

\newcommand{\code}[1]{\texttt{#1}}

\DeclareMathOperator*{\spa}{span}

\newcommand*\diff{\mathop{}\!\mathrm{d}}

\makeatletter
\newcommand*{\notableentry}{%
  \kern-\tabcolsep
  \stepcounter{NoTableEntry}%
  \vadjust pre{\zsavepos{\theNoTableEntry t}}
  \vadjust{\zsavepos{\theNoTableEntry b}}
  \zsavepos{\theNoTableEntry l}
  \raisebox{%
    \dimexpr\zposy{\theNoTableEntry b}sp
    -\zposy{\theNoTableEntry l}sp\relax
  }[0pt][0pt]{%
    \color{black}%
    \setlength{\unitlength}{1pt}%
    \edef\w{%
      \strip@pt\dimexpr\zposx{\theNoTableEntry r}sp%
      -\zposx{\theNoTableEntry l}sp\relax
    }%
    \edef\h{%
      \strip@pt\dimexpr\zposy{\theNoTableEntry t}sp%
      -\zposy{\theNoTableEntry b}sp\relax
    }%
    \ifdim\w pt=0pt 
    \else
      \begin{picture}(0,0)%
        \edef\x{%
          \noexpand\put(0,0){\noexpand\line(\w,\h){\w}}%
          \noexpand\put(0,\h){\noexpand\line(\w,-\h){\w}}%
        }\x
      \end{picture}%
    \fi
  }%
  \hspace{0pt plus 1filll}%
  \zsavepos{\theNoTableEntry r}
  \kern-\tabcolsep
}

\title{Approximating optimal feedback controllers of finite horizon control problems using hierarchical tensor formats}

\renewcommand\@date{{%
  \vspace{-\baselineskip}%
  \large\centering
  \begin{tabular}{@{}c@{}}
    Mathias Oster\textsuperscript{1} \\
    \normalsize \texttt{oster@math.tu-berlin.de} 
  \end{tabular}%
  \quad 
  \begin{tabular}{@{}c@{}}
    Leon Sallandt\textsuperscript{1} \\
    \normalsize \texttt{sallandt@math.tu-berlin.de} 
  \end{tabular}
  \quad 
  \begin{tabular}{@{}c@{}}
    Reinhold Schneider\textsuperscript{1} \\
    \normalsize \texttt{schneidr@math.tu-berlin.de} 
  \end{tabular}

  \bigskip

  \textsuperscript{1}Technische Universit\"at Berlin

  \bigskip

  \today
}}

\begin{document}

\maketitle

\begin{abstract}
Controlling systems of ordinary differential equations (ODEs) is ubiquitous in science and engineering.
For finding an optimal feedback controller, the value function and associated fundamental equations such as the Bellman equation and the Hamilton-Jacobi-Bellman (HJB) equation are essential. 
The numerical treatment of these equations poses formidable challenges due to their non-linearity and their (possibly) high-dimensionality.

In this paper we consider a finite horizon control system with associated Bellman equation.
After a time-discretization, we obtain a sequence of short time horizon problems which we call local optimal control problems.
For solving the local optimal control problems we apply two different methods, one being the well-known policy iteration, where a fixed-point iteration is required for every time step.
The other algorithm borrows ideas from Model Predictive Control (MPC), by solving the local optimal control problem via open-loop control methods on a short time horizon, allowing us to replace the fixed-point iteration by an adjoint method.

For high-dimensional systems we apply low rank hierarchical tensor product approximation/tree-based tensor formats, in particular  tensor trains (TT tensors) and 
multi-polynomials, together with high-dimensional quadrature, e.g. Monte-Carlo.

We prove a linear error propagation with respect to the time discretization and give numerical evidence by controlling a diffusion equation with unstable reaction term and an Allen-Kahn equation.
\end{abstract}

\section{Introduction}

We consider a finite horizon optimal control problem on a $d$-dimensional state space. In our applications the constraining dynamical systems are spatial discretizations of partial differential equations (PDE) and thus $d$ can become large. Opposed to open-loop controls, feedback laws provide controls for all initial values simultaneously and in many cases produce stabilizing controls even under perturbations. 
For finding an optimal feedback law, a mapping from the state to the optimal control is needed.
In order to allow online usage of the feedback law fast evaluation of this map is critical.

One popular approach for finding an optimal feedback law is to calculate the value function $v^*:\mathbb R^d\times [0,T]\to\mathbb R$, which is a mapping from the state space to the real numbers.
This function obeys a fundamental equation, the Bellman equation \cite{bellman1957dynamic, DeterministicHJB}.
For given time-points the Bellman equation is solved recursively by a sequence of optimal control problems with short time horizon $\tau \ll T$, where $\tau$ is the step size in a time discretization of the value function.

We treat these time-local optimal control problems using two different methods, one being an open-loop/adjoint approach and the other being the policy iteration algorithm \cite{howard1960dynamic}.

The open-loop ansatz is motivated by the observation that the value function can be computed point-wise using traditional methods from optimal control such as open-loop control/adjoint methods \cite{Pontryagin}.

In order to obtain such point-evaluations of the value function, the optimal control problem has to be solved for an initial value.
The dimension of this optimal control problem scales with the time-horizon and for every gradient step the ODE has to be solved in a forward and a backward way.
In that sense, computing an optimal control and thus the value function can become expensive for long time horizons.
However, due to short time horizons $\tau$ of the local optimal control problems the open-loop approach is efficient for our approach.

Using these point measurements we perform an interpolation/regression to obtain the value function for every initial state.
This yields a backwards iteration w.r.t. time.
Note that in our numerical tests computing the local open-loop control problems is the numerical bottleneck and this step can be parallelized perfectly.

In the second approach we use the policy iteration algorithm combined with a regression to solve the local optimal control problems.
This yields a similar backwards iteration with a nested fixed-point equation.
In this approach several regression problems have to be solved for every time-point.
However, generating samples is less expensive, shifting the numerical bottleneck from generating the samples to solving the equations.

%

As the dimension of the state space increases, traditional methods to represent the function, like Galerkin approximation by splines, finite elements or multi polynomials suffer from the so-called curse of dimensionality, which means that their complexity increases exponentially with the dimension of the state space.
To alleviate this problem we use a non-linear model class, namely low rank hierarchical tensor product approximation/tree-based tensor formats, in particular  tensor trains (TT tensors) and multi-polynomials, together with high dimensional Monte-Carlo quadrature.
By this approach we compute a low-fidelity/complexity representation of the required optimal feedback control such that it is easily applied in online computations.
We prove a linear error propagation with respect to the time discretization in Theorem \ref{thm:error_propagation}. 

Finally, we provide numerical evidence for both approaches by testing a feedback law computed for a diffusion equation with unstable reaction term as well as a Allen-Kahn equation.
Moreover, we propose possible extensions of the methods where more information about the system is used.
\subsubsection*{Previous Work}
The Bellman equation, also known as dynamic programming, was introduced in the 1950s by R. Bellman \cite{bellman1966dynamic}.
Note that in addition to the Bellman equation, there exists an infinitesimal version, the Hamilton-Jacobi-Bellman (HJB) equation, see e.g. \cite{DeterministicHJB, falcone2013semi}.
Both equations can be approximated by the policy iteration \cite{howard1960dynamic}, which is a widely used tool in optimal feedback control, see e.g. \cite{pol_approx_kunisch, alla2015efficient} Other popular methods for solving the HJB equation are semi-Lagrangian methods \cite{SemiLagranigian,SemiLagrangianStochastic,Falcone1987}, Domain splitting algorithms \cite{FALCONESplitting}, variational iterative methods \cite{VIM}, data based methods with Neural Networks \cite{DataHJB}, actor-critic methods \cite{zhou2021actor}, tree-based methods \cite{ALLA2020192} and tropical algorithms \cite{maxplus_det,maxplus_stoch}.

Simultaneously to the work of Bellman, Pontryagin developed a set of necessary conditions for optimal problems \cite{Pontryagin}, the so-called Pontryagin maxiumum principle (PMP)
This approach naturally leads to adjoint methods, from where feedback control is not immediately applicable.

Despite its open-loop nature, the PMP has already been used to find feedback controls, see e.g. \cite{beeler2000feedback, kang2017mitigating, nakamura2019adaptive, azmi2020optimal}.
However, in contrast to our open-loop approach, these approaches do not use the Bellman equation and thus the PMP approach can suffer from long time-horizons.

Another approach, closely related to our open-loop ansatz, where the concept of adjoint methods and feedback control are combined is Model Predictive Control (MPC).
Model predictive control originates for simple applications in the 1970s for stabilizing linear systems, see i.e. \cite{kleinman1970easy}, and later reformulated for non-linear systems, see e.g. \cite{grune2017nonlinear}.
For surveys on MPC w.r.t. theoretical results and industrial applications we refer to \cite{garcia1989model, qin2003survey}.
The MPC approach replaces the optimal control problem with short and overlapping time horizon and obtains a feedback control by solving this short time horizon problem, to obtain a sub-optimal online controller.
However, due to its short time horizon, this optimal control problem can be solved in real-time.
Our open-loop method can be characterized as a MPC approach with optimal final condition, leading to an optimal controller.
This combination of MPC and the Bellman equation has already been considered in the context of reinforcement learning \cite{atkeson1994using,zhong2013value}.
In these applications, low-dimensional systems are considered.

The second approach we consider in this paper is an application of the policy iteration algorithm.

In order to control higher-dimensional systems, memory efficient methods have to be applied to circumvent the curse of dimensionality.
Hierarchical tensor products such as tensor trains have also been recently used in this context.
Rooted in quantum physics under the name \emph{matrix product states}, tensor trains have been introduced to the mathematical community in \cite{oseledets2011tensor} to tackle the curse of dimensionality.
Note that tensor trains are a special case of a more general framework - hierarchical tensor networks.
These networks have been developed in \cite{Hackbusch-2010} where a well-founded mathematical framework is introduced.
For surveys and more details, see \cite{Hackbusch-Acta, Hackbusch2014,Legeza-Schneider,Bachmayr-Uschmajew-Schneider}.
Tensor trains have already been applied to represent the value function of infinite horizon control problems \cite{dolgov2019tensor,oster2019approximating} 
and stochastic control problems \cite{horowitz2014linear, fackeldey2020approximative, gorodetsky2018high}.

Instead of this ansatz space it is possible to use other methods from machine learning such as neural networks, see e.g.,  \cite{darbon2020overcoming, nusken2020solving,ito2020neural}, kernel methods or sparse polynomials, as it is done in \cite{azmi2020optimal}.

Finally, our regression/Least-Squares-based approach is closely related with empirical risk minimization \cite{vapnik1992principles, steinwart} and classical machine learning tasks with one fundamental twist.
In classical tasks from statistical and machine learning the data is noisy and samples are prescribed and biased.
In our approach, we are able to compute the data with close to arbitrary precision due to the adjoint method approach and advanced ODE solvers.
Moreover, we are able to generate or data points arbitrarily.
Due to the precision of the adjoint approach, computing the data can be understood as a so-called high-fidelity computation.
In contrast to that, our approximation of the value function can be seen as a low-fidelity representation of the value function, enabling real-time evaluations of the gradient and thus of the (approximative) optimal control.

While the generalization to controlling stochastic systems is straight-forward in the case of the policy iteration, the generalization of the open-loop approach to stochastic control is more complex.
Here, possible approaches are the solution of forward-backward stochastic differential equations (FBSDE) \cite{bouchard2004discrete, gobet2005regression}, for an optimal control context see e.g. \cite{pham2005some}.
We want to highlight recent groundbreaking successes in the treatment of FBSDEs using deep learning, see e.g. \cite{weinan2017deep}.
It was also shown that tensor trains are efficiently applicable in this context \cite{richter2021solving}.
\subsubsection*{Structure}
The rest of the paper is organized as follows.
In the preliminaries Section $2$ we introduce the optimal control problem, and fundamental concepts such as the value function and the Bellman equation.

Section $3$ is devoted to reinterpreting the Bellman equation as an optimal control problem and the corresponding open-loop ansatz.
In the following sections we treat the Bellman equation via the policy iteration algorithm.
Then, we clarify the numerical feasibility by introducing the regression method in combination with our ansatz set, the tensor trains.
In Section \ref{sect:compare} we give several ideas on how to improve the approach and
in the final section we give numerical evidence.

\section{The Optimal Control Problem}
In this section we formulate the optimal control problem in open-loop and closed-loop form. 
Furthermore, we define basic notions of control theory and formulate governing equations.

We consider a deterministic, finite time horizon optimal control problem of the following form.
For $x \in \Omega \subset \mathbb R^d$ minimize w.r.t. $u \in L^2(0, T; \mathbb R^m)$ the cost functional $\mathcal J: [0, T] \times \Omega \times L^2(0, T; \mathbb R^m) \to \mathbb R$, defined as,
\begin{equation}\label{eq:cost}
   \mathcal J(t_0, x, u(\cdot)) := \int_{t_0}^T c(t, y(t)) + u(t)' R(t) u(t) \diff t + c_T(y(T)),  
\end{equation} 
where
\begin{align}
    \dot y(t) &= f(t, y(t)) + g(t, y(t)) u(t) \label{eq:rhs} \\
    y(0) &= x.
\end{align}
We assume $c:[0,T] \times \Omega \to \mathbb{R}_{\geq 0}$ and $c_T:\Omega \to \mathbb R_{\geq 0}$ are non-negative, coercive and smooth.
Further, let $R:[0,T] \to \mathbb R^{m,m}$ be a family of positive definite matrices continuous in time.
Note that $u(t)'$ denotes the transpose of $u(t)$. 
For initial data $x \in \Omega$ and fixed control $u(\cdot) \in \mathcal A$, we denote by $y^x(t, u) \in \Omega$ the evaluation of the trajectory at time $t$. If the context is clear we just write $y(t)$. 
We further assume that $f: [0, T] \times \Omega \to \Omega$ and $g: [0,T] \times \Omega \to \mathbb R^{n,m}$ are (non linear) smooth functions. 
We define the value function as
\begin{equation}
	v^*: [0,T] \times \Omega \to \mathbb R, \quad (t, x) \mapsto \inf_{u \in L^2(0,T;\mathbb R^m)} \mathcal J(t, x, u)
\end{equation}
and the corresponding Bellman equation takes the following form.
\begin{theorem}\cite{DeterministicHJB, Bardi1997}
Set $\ell(t, x, u) = c(t, x) + u' R(t) u$.
Then for all $x \in \Omega$ and $0 \leq t_0 < t_1 \leq T$ we have
\begin{equation}\label{eq:bellman}
    v^*(t_0, x) = \inf_{u \in L^2(t_0, t_1; \mathbb R^m)} \Big[ \int_{t_0}^{t_1} \ell(y(t), u(t)) \diff t + v^*(t_1, y(t_1)), \Big]
\end{equation}
where $y(\cdot)$ is the trajectory corresponding to \eqref{eq:rhs} with control $u$ and initial value $y(t_0) = x$.
\end{theorem}
We further have the following HJB equation
\begin{theorem}\cite{DeterministicHJB, Bardi1997}
Assume there are $\sigma, \delta \geq 1$ with $\sigma<\delta$, $\ell_0>0$ and for every compact $K\subset \mathbb R^d$ some $f_K>0$ such that
\begin{align*}
    \|f(x,u)\| &\leq f_K(1+\|x\|^\sigma) \quad \forall (x,u)\in K\times \mathbb R^m,\\
    |\ell(x,u)| &\geq \ell_0\|a\|^\delta \quad \forall (x,u)\in \mathbb R^d\times \mathbb R^m.
\end{align*} Then the value function $v^*(t, x)$ is the unique viscosity solution of
\begin{equation}
    \frac {\partial}{\partial t}v^*(t, x) + \sup_{u \in \mathbb R^m} \Big [ -\nabla v^*(t, x) \cdot (f(t, x) + g(t, x) u) - \ell(t, x, u) \Big] = 0
\end{equation}
with final condition $v(T, \cdot) = c_T(\cdot)$.
\end{theorem}
In the following, we refer to the collection \eqref{eq:cost} and \eqref{eq:rhs} as the open-loop control problem.
In contrast to that, in the closed-loop formulation, the parameter $u(t)$ is replaced by a mapping $\alpha : \Omega \to \mathbb R^m$, such that the closed-loop system takes the form.
\begin{align}\label{eq:closed_loop}
    \dot y = f(t, y)+g(t, y)\alpha(t, y), \quad y(t_0) = x
\end{align}
for any policy $\alpha:[0, T] \times \Omega\to \mathbb R^m$  continuous on $[0,T]$ and Lipschitz in $\Omega$ and denote the policy evaluation function
\begin{equation}\label{eq:feedback_cost}
   \mathcal J^\alpha(t_0, x) := \int_{t_0}^T c(t, y^\alpha(t)) + \alpha(x(t))' R(t) \alpha(x((t)) \diff t + c_T(y^\alpha(T)).
\end{equation}
Here, $y^\alpha(t)$ is the trajectory of the closed-loop system \eqref{eq:closed_loop} evaluated at time $t$.
We refer to the collection \eqref{eq:feedback_cost}, \eqref{eq:closed_loop} as the closed-loop problem.
If the value function is known and differentiable, we can use it to explicitly compute an optimal feedback law.
\begin{theorem}\cite{DeterministicHJB}\label{thm:optimality_condition}
An optimal feedback control is given by
\begin{equation}
    \alpha^*(t, x) = - \frac 1 2 R^{-1} g(t, x)' \nabla v^*(t, x),
\end{equation}
if $\nabla v^*$ exists.
\end{theorem}
For $\tau\geq0$ we define the flow  $\Phi_{\tau}^\alpha:[0,T] \times \Omega\to [\tau,T+\tau] \times  \Omega$ such that $\Phi_{\tau}^\alpha(t_0, x) = (t_0 + \tau, y(t_0+\tau))$, where $y(t_0+\tau)$ is the evaluation of the trajectory at time $\tau+t_0$ with initial condition $y(t_0) = x$ w.r.t \eqref{eq:closed_loop}.

\section{Reformulation as a series of Open-Loop control Problems}
In this section we reinterpret the Bellman equation as a series of open-loop control problems.
These open-loop control problems are defined on a small subset of $[t_l, t_{l+1}) \subset [0, T]$, which is why we refer to them as local open-loop control problems.
To this end, we consider a discrete set of time points $0 = t_1 < \dots < t_L = T$.
We notice that the Bellman equation has essentially the same structure as the finite horizon control problem and that we can interpret it as such.
For convenience we repeat the Bellman equation below
\begin{equation}\label{eq:bellman_repeat}
    v^*(t_l, x) = \inf_{u \in L^2(t_l, t_{l+1}; \mathbb R^m)} \Big[ \int_{t_l}^{t_{l+1}} \ell(y(t), u(t)) \diff t + v^*(t_{l+1}, y(t_{l+1})) \Big],
\end{equation}
and notice that $v^*(t_{l+1}, y(\cdot))$ can be interpreted as a final condition and $\ell(y(\cdot), u(\cdot))$ as the running cost with governing ODE system \eqref{eq:rhs}.

The local open-loop control problem can be solved by adjoint methods. Recall that the Pontryagin Maximum Principle (PMP) gives as necessary conditions for a minimal control
\begin{align*}
    \dot y(t) = f(t,y)+g(t,y)\lambda (t), & & y(0)=x\\
    \dot \lambda (t)  = - \partial_y H(y,u,\lambda) & & \lambda(T) = \partial_y c_T(y(T))\\
    H(y,u,\lambda) = \min_{\tilde u} H(y,\tilde u ,\lambda)
\end{align*}
where $H(y,u,\lambda) = \ell(y,u) + \lambda^T (f(t,y)+g(t,y)u)$, \cite{Pontryagin, Dontchev1995}.
After introducing a time discretization, this system is solved by standard gradient decent methods, see e.g. \cite{herzog2010algorithms}.
Note that in this formulation the gradient of the final condition appears, which means that by setting the value function to be the final condition we might run into regularity issues.
However, we ignore this problem for now and assume that the value function is differentiable.
We alleviate this problem in Section \ref{sect:vmc}.

Traditionally, open-loop control methods suffer from long time horizons. However, due to our time-discretization, we can assume that the time-horizon is small.
This results in the following algorithm for the continuous case.

\begin{algorithm}[H]\label{algo:pol_it_cont}
\SetAlgoLined
\caption{Backwards solution to the Bellman equation - continuous case}\label{opt_con:algo:pol_it}
\SetKwInOut{Input}{input}\SetKwInOut{Output}{output}
\SetKwInOut{Output}{output}\SetKwInOut{Output}{output}
\Input{Time points $0 = t_0 < \dots < t_L < T$.}
\Output{The value function $v^*$ evaluated at the time points}
Set $v^*(t_L, \cdot) = c_T(\cdot)$

\For{$l = L-1$ to $0$}{
    Calculate $v^*(t_l, \cdot)$ by solving the local open-loop control problem \eqref{eq:bellman_repeat} with final condition $v^*(t_{l+1}, \cdot)$.\label{algo:line_pol_it_cont}
}
\end{algorithm}
We observe that within this formulation we have not done any discretization other than setting time points where we want to compute the value function.
The next step is usually a time-discretization of the underlying ODE \eqref{eq:rhs}, such that the optimal control problem can be solved.
This time-discretization does not have to be on the same time-grid as the time-points where the value function shall be evaluated.
In fact, it can be finer.
In this case we use linear interpolation of the value function to obtain the value function at the intermediate points
\begin{equation}\label{eq:linear_interpolation}
   v^*(t, x) \approx \frac{t_{l+1} - t}{\tau} v^*(t_l, x) + \frac{t -t_l}{\tau} v^*(t_{l+1}, x) \text{ for } t \in [t_l, t_{l+1}), x \in \Omega.
\end{equation}
Note that this interpolation does not affect Algorithm \ref{algo:pol_it_cont} as in this algorithm the value function is only evaluated at the time-points $t_l$.
This is only relevant when the value function is used to compute a feedback-law via Theorem \ref{thm:optimality_condition}.
We also make use of this interpolation for the policy iteration approach in Section \ref{sect:polit}.

Assuming an equidistant discretization of the ODE, the dimension of the local open-loop control problems increases only linearly in time and does not increase with the spatial dimension.
However, the dimension of the regression problem is problematic and is covered in Section \ref{sect:TT}.
Note that solving the local optimal control problems for every initial value is computationally not feasible.
We instead draw samples $x_i \in \Omega$ and use a regression type approach as described in Section \ref{sect:vmc}.
\section{Policy Iteration approach}\label{sect:polit}
Another popular approach is the Policy iteration algorithm. This Newton-type method is an iterative scheme that alternates between solving a linearized Bellman equation and improving a policy.

As in the previous chapter, we use the same time subdivision $0 = t_0 < \dots < t_L = T$.
We use the policy iteration to approximate the value function by solving the local optimal control problems on our time-grid in the following way.
For a fixed policy $\alpha$ defined on $[t_l, t_{l+1})$, the corresponding linearized Bellman equation takes the linearized form
\begin{align}
    v^\alpha(t_l, \cdot) &= \int_{t_l}^{t_{l+1}} \ell_{t - t_l}^{\alpha}(t_l, \cdot)dt + v^*(t_{l+1}, \Phi_{\tau}^{\alpha} (t_{l+1}, \cdot)), \quad v^\alpha(t_{l+1}, \cdot) = v^*(t_{l+1}, \cdot). \label{eq:coupled_hjb_local}
\end{align}
Note that plugging in the optimal policy recovers the Bellman equation \eqref{eq:bellman}.
In this formulation the policies are depending on the time $t$.
However, if we only compute $v^\alpha(t_l, x)$ for every $l$, the policy using the optimality condition in Theorem \ref{thm:optimality_condition} cannot be evaluated at times other than $t_l$.
Thus, we again make use of interpolation between the time-points
\begin{equation}\label{eq:linear_interpolation_valpha}
   v^\alpha(t, \cdot) = \frac{t_{l+1} - t}{\tau} v^\alpha(t_l, \cdot) + \frac{t -t_l}{\tau} v^*(t_{l+1}, \cdot) \text{ for } t \in [t_l, t_{l+1}).
\end{equation}
The policy iteration algorithm is then given by Algorithm \ref{algo:pol_it_local}.
\newline
\begin{algorithm}[H]
\SetAlgoLined
\caption{Policy Iteration for approximating the local optimal control problem.}\label{algo:pol_it_local}
\SetKwInOut{Input}{input}\SetKwInOut{Output}{output}
\SetKwInOut{Output}{output}\SetKwInOut{Output}{output}
\Input{A Policy $\alpha_{0}$,  $0 \leq t_l < t_{l+1} \leq T$ and an approximation of $v^*(t_{l+1}, \cdot)$, denoted by $\hat v(t_{l+1}, \cdot)$}
\Output{An approximation of $v^*(t_l, \cdot)$ and $\alpha^*(t_l, \cdot)$, denoted by $\hat v(t_{l+1}, \cdot)$ and $\hat \alpha(t_{l+1}, \cdot)$.}
Set $k = 0$.

\While{not converged}{
  Solve the linear equation
\begin{equation}\label{eq:pol_it_time}
     v_{k+1}(x) = \int_{t_l}^{t_{l+1}} \ell_{t - t_l}^{\alpha_k}(t_l, x) dt + \hat v(t_{l+1}, \Phi_{\tau}^{\alpha_k} (t_l, x))
\end{equation}
and use linear interpolation between $v_{k+1}$ and $\hat v(t_{l+1}, \cdot)$ as in \eqref{eq:linear_interpolation_valpha} to obtain $v_{k+1}(t, \cdot)$ for $t \in [t_l, t_{l+1})$.
Then update the policy according to
\begin{equation*}
\alpha_{k+1}(t, x) = - \frac 1 2 B(t)^{-1} g(t, x)^T \nabla v_{k+1}(t, x), \quad t \in [t_l, t_{l+1}).
\end{equation*}

$k = k+1$.
}
Set $\hat v(t, \cdot) = v_{k}(t, \cdot)$ and $\hat \alpha(t, \cdot) = \alpha_{k}
$ for $t \in [t_l, t_{l+1})$.
\end{algorithm}

In \cite{Puterman,Saridis} the convergence of the exact Policy iteration is analyzed and in \cite{Beard} extended to Galerkin methods. The convergence and error estimates are still open problems in the Least-Squares setting. Due to the policy update where the gradient of $v$ appears within the fixed-point iteration, this method can be prone to overfitting.
We address this issue by adding a regularization term, as described in \eqref{eq:regularizer}.

Note that the right-hand-side of \eqref{eq:pol_it_time} can be computed point wise without knowledge of the underlying dynamical system. Only a black-box solver for $\Phi_{t_l}^{\alpha}$ and a function for evaluation of $\ell_{t_l}^{\alpha}$ is needed. We also observe, that the time-discretization of the value function does not not necessarily have to be the same as the time discretization of the ODE.
The above algorithm is used to solve the Bellman equation on the complete time frame.


\section{Formulation of the Regression Problem}\label{sect:vmc}
For solving (\ref{eq:bellman_repeat}, \ref{eq:pol_it_time}), computing the value function for every point $x$ is not feasible.
Moreover, we cannot expect to have access to the exact value function $v^*(t_{l+1}, \cdot)$ as final condition, but only an approximation. We compute the approximation of the value function by a regression ansatz, which is very similar for both approaches. We first describe the procedure for the local open-loop ansatz and thereafter comment on the policy iteration approach.

We replace the value function as a final condition with an approximation $\hat v_{l+1} (\cdot)$ that we computed in the step before, i.e. we replace the local optimal control problem \eqref{eq:bellman_repeat} by
\begin{equation}\label{eq:local_opt_tilde_v}
    \tilde v_l(x) = \inf_{u \in L^2(t_l, t_{l+1}; \mathbb R^m)} \int_{t_l}^{t_{l+1}} \ell(y(t), u(t)) \diff t + \hat v_{l+1}(y(t_{l+1})).
\end{equation}
Note that we write $\tilde v_l$ on the l.h.s. and $\hat v_{l+1}$ on the r.h.s. of the equation.

We do not use the same notation because we are not be able to find the exact minimizer $\tilde v_l$ for every initial value $x$.
Instead we find an approximation, which we again denote by $\hat v_l$.
In order to find this approximation we observe that $\tilde v_l \in L^2(\Omega) =: V$.
Choosing an ansatz space $\mathcal M \subset V$ and we seek an approximation in the sense of a Least-Squares method
\begin{equation}
    \inf_{v \in \mathcal M} \| v - \tilde v_l \|_V^2.
\end{equation}
As we cannot compute the norm, we replace this equation by a discrete version using empirical risk minimization \cite{cucker}.
To this end we consider a set of samples $\{ x_1, \dots x_J \} \subset \Omega$ and compute $\tilde v_l (x_i)$ for all $1 \leq j \leq J$,
which is a regression problem.
In particular we set
\begin{equation}
    \hat v_l = \argmin_{v \in \mathcal M} \frac{1}{J} \sum_{j = 1}^J |v(x_j) - \tilde v(x_j) |^2.
\end{equation}
In this paper we consider tensor networks as ansatz space.
Error estimates w.r.t. the number of samples are given in \cite{eigel2020adaptive} for our ansatz space.
At this point we notice that the problem of a potentially irregular value function can be solved here.
By replacing the final condition $v^*(t_l, \cdot) \in L^2(\Omega)$ by the approximation $\hat v_l \in \mathcal M$ we can enforce a smooth final condition by choosing an ansatz space $\mathcal M$ consisting of smooth functions. 
This results in the following discrete version of Algorithm \ref{algo:pol_it_cont}.

\begin{algorithm}[H]\label{algo:pol_it_disc}
\SetAlgoLined
\caption{Backwards solution to the Bellman equation - discrete case}\label{opt_con:algo:pol_it_discr}
\SetKwInOut{Input}{input}\SetKwInOut{Output}{output}
\SetKwInOut{Output}{output}\SetKwInOut{Output}{output}
\Input{Time points $0 = t_0 < \dots < t_L < T$ and samples $x_1, \dots, x_J \in \Omega$}
\Output{An approximation $\hat v_l$ of the value function $v^*(t_l, \cdot)$ evaluated at the time points}
Set $\hat v_L = c_T$

\For{$l = L-1$ to $0$}{
    Calculate $\tilde v_l(x_j)$ by solving the local open-loop control problem \eqref{eq:local_opt_tilde_v} with initial values $x_j$ and final condition $\hat v_{l+1}$.\label{algo:line_pol_it_discr}
    
    Compute $\hat v_l$ by solving the regression problem
    \begin{equation}\label{eq:algo_regression}
        \hat v_l = \argmin_{v \in \mathcal M} \frac{1}{J} \sum_{j = 1}^L |v(x_j) - \tilde v(x_j) |^2.
    \end{equation}
}
\end{algorithm}
\begin{remark}
    Up until now we have not specified the ansatz set $\mathcal M$. Note that here, any ansatz set where \eqref{eq:algo_regression} is computable can be used.
    In low dimension, traditianal ansatz sets such as finite elements, splines or polynomial ansatz spaces are feasible.
    In higher dimension, other ansatz sets such as neural networks or tensor networks, as considered here, are possible.
    Note that computing this algorithm consists of $L$ equations that have to be solved where $J$ samples have to be generated.
    This formulation is especially viable if solving the regression problems is particularly expensive and if the local open-loop control problems can be solved efficiently.
\end{remark}
For the policy iteration approach the discrete version of the algorithm is very similar.
Here, we simply replace \eqref{eq:pol_it_time} by a discrete version 
\begin{equation}\label{eq:discrete_pol_it}
     v_{k+1} = \argmin_{v \in \mathcal m}  \frac 1 J \sum_{j=1}^J \int_{t_l}^{t_{l+1}} | v(x_j) - \ell_{t - t_l}^{\alpha_k}(t_l, x_j) dt - \hat v(t_{l+1}, \Phi_{\tau}^{\alpha_k} (t_l, x_j))|^2.
\end{equation}
We finalize this section by giving an error propagation under the assumption that we can bound the difference between $\tilde v$ and $\hat v$ in the $L^\infty(\Omega)$ norm.

We first prove that for optimal control problems of similar form, the value functions are similar.
\begin{lemma}\label{lem:v_cont_wrt_final}
Let $v_1, v_2$ be value functions to optimal control problems of the form
\[ v_1 (x) = \inf_u \int_{t_0}^{t_1} \ell(y, u) dt + c_1(y(t_1)), \quad v_2(x) = \inf_u \int_{t_0}^{t_1} \ell(y, u) dt + c_2(y(t_1)), \]
with the same underlying ODEs, $y(\cdot)$ be the solution to the ODEs with initial condition $x$ and
$| c_1(x) - c_2(x) | \leq \delta$ for all $x$. Then
\begin{equation*}
    |v_1(x) - v_2(x) | \leq \delta
\end{equation*}
as well.
\end{lemma}
The proof is given in the appendix. This Lemma allows us to prove an error propagation within our discrete algorithm.
The main idea is that the solution $\tilde v_l$ to the local optimal control problem \eqref{eq:local_opt_tilde_v} can be seen as an intermediate between the exact value function $v^*(t_l, \cdot)$ and the computable approximation $\hat v_l$.
\begin{theorem}\label{thm:error_propagation}
By $\tilde v$ denote the solution to the local optimal control problem \eqref{eq:local_opt_tilde_v} with terminal condition $\hat v_{l+1}$.
Assume that $ |\tilde v_l(x) - \hat v_l(x)| \leq \delta$ for all $x \in \Omega$ and all $1 \leq l \leq L$.
Then
\[ \| \hat v_{L - l}(\cdot) - v^*(T - l \tau, \cdot) \|_{L^\infty (\Omega)}  \leq l \delta \]
for all $0 \leq l \leq L$.
\end{theorem}
\begin{proof}
We prove the theorem inductively. First note that we have $v^*(T, \cdot) = \hat v_L(\cdot)$. Thus,
\begin{equation*}
    \| \hat v_{L - 1} (\cdot) - v^*(T - \tau, \cdot) \|_{L^\infty(\Omega)} \leq \delta 
\end{equation*}
by assumption. 

Now assume that 
\begin{equation}\label{eq:proof3}
\| \hat v_{L - l}(\cdot) - v^*(T- l \tau, \cdot) \|_{L^\infty(\Omega)} < l \delta.
\end{equation}
Finally, we have
\begin{align}
    &\quad \| \hat v_{L - (l+1)}(\cdot) - v^*(T - (l+1) \tau, \cdot) \|_{L^\infty(\Omega)} \\
    &\leq  \underbrace{\| \hat v_{L - (l+1)}(\cdot) - \tilde v_{L -(l+1)}(\cdot) \|_{L^\infty(\Omega)}}_{\leq \delta \text{ by assumption}} + \underbrace{\| \tilde v_{L - (l+1)} (\cdot) - v^*(T - (l+1) \tau, \cdot) \|_{L^\infty(\Omega)}}_{\leq l \delta \text{ by \eqref{eq:proof3} plugged into Lemma \ref{lem:v_cont_wrt_final}}} \\
    &\leq (l+1) \delta.
\end{align}
This finishes the proof.
\end{proof}
This theorem yields an error bound in the $L^\infty(\Omega)$ norm.
Using a regression approach, however does only yield bounds in the $L^2(\Omega)$ norm.
Here, we make use of our finite-dimensional model set.
As the model set $\mathcal M$ is continuously embedded into its span, i.e.  $ \mathcal M \hookrightarrow \spa(\mathcal M)$, this space is a finite-dimensional, and thus closed, subspace of $L^2(\Omega)$.
Due to the boundedness of $\Omega$ and the regularity of the ansatz functions this subspace can also be equipped with the $L^\infty(\Omega)$ norm and as a finite dimensional space these norms are equivalent, which means that a bound in the $L^2(\Omega)$ norm is in fact also a bound in the $L^\infty(\Omega)$ norm.

\section{Representation of the Value Function}\label{sect:TT}
The question of how to represent the approximation of the value function $\hat v_l \in \mathcal M$ and how to find it within the ansatz space is remaining. For the sake of readability we drop the subscript $l$ in this section.

Traditional methods such as finite elements, splines or multi-variate polynomials leads to a computational complexity that scales exponentially in the state space dimension $d$.
However, interpreting the coefficients of such ansatz functions as entries in a high-dimensional tensor allows us to use tensor compression methods to reduce the number of parameters.
To this end, we define a set of functions $\{ \phi_1, \dots, \phi_m \}$ with $\phi_i : \mathbb R \to \mathbb R$ ,
e.g. one-dimensional polynomials or finite elements.
The approximation $\hat v : \mathbb{R}^d \rightarrow \mathbb{R}$ takes the form
\begin{equation}
\label{eq:V c}
    \hat v(x_1, \dots, x_d) = \sum_{i_1 = 1}^m \dots \sum_{i_d = 1}^m c_{i_1, \dots, i_d} \phi_{i_1}(x_1) \cdots \phi_{i_d} (x_d).
\end{equation}
Note that this is a standard formulation for finite elements or multivariate polynomial bases.
For the sake of simplicity we choose the set of ansatz functions to be the same in every dimension. 
The coefficient tensor $c \in \mathbb R^{m \times m \times \dots \times m} \equiv \mathbb R^{m^d}$ suffers from the curse of dimensionality since the number of entries increases exponentially in the dimension $d$.
In what follows, we review the tensor train format to compress the tensor $c$.

For the sake of readability we will henceforth write $c_{i_1, \dots, i_d} = c[i_1, \dots, i_d]$ and represent the contraction of the last index of a tensor $w_1 \in \mathbb R^{r_1 \times m \times r_2}$ with the first index of another tensor $w_2 \in \mathbb R^{r_2 \times m \times  r_3}$ by
\begin{subequations}
\begin{align}
    w &= w_1 \circ w_2 \in \mathbb R^{r_1 \times m \times m \times r_3}, \\
    w[i_1, i_2, i_3, i_4] &= \sum_{j = 1}^{r_2} w_1[i_1, i_2, j] w_2[j, i_3, i_4].
\end{align}
\end{subequations}
In the literature on tensor methods, graphical representations of general tensor networks are widely used.
In these pictorial descriptions, the contractions $\circ$ of the component tensors are indicated as edges between vertices of a graph.
As an illustration, we provide the graphical representation of an order-$4$ tensor and a tensor train representation (see Definition \ref{def:tensor_train} below) in Figure \ref{TT:fig:hosvd}.
\begin{figure}[h!]
    \centering
    \begin{tikzpicture}
        \begin{scope}[every node/.style={draw,  fill=white}]
        \node (A1) at (0,0) {$u_1$}; 
        \node (A2) at (1.25,0) {$u_2$}; 
        \node (A3) at (2.5,0) {$u_3$}; 
        \node (A4) at (3.75,0) {$u_4$}; 
        
        \node (A0) at (-3,0) {$c$}; 
        \end{scope}
        \node (B2) at (-1,0) {$=$}; 
        \begin{scope}[every edge/.style={draw=black,thick}]
        	\path [-] (A1) edge node[midway,left,sloped] [above] {$r_1$} (A2);
        	\path [-] (A2) edge node[midway,left,sloped] [above] {$r_2$} (A3);
        	\path [-] (A3) edge node[midway,left,sloped] [above] {$r_3$} (A4);
        	
        	\path [-] (A1) edge node[midway,left] [right] {$m$} +(-90:1);
        	\path [-] (A2) edge node[midway,left] [right] {$m$} +(-90:1);
        	\path [-] (A3) edge node[midway,left] [right] {$m$} +(-90:1);
        	\path [-] (A4) edge node[midway,left] [right] {$m$} +(-90:1);
        	
        	\path [-] (A0) edge node[midway,left] [right] {$m$} +(-90:1);
        	\path [-] (A0) edge node[midway,left] [above] {$m$} +(180:1);
        	\path [-] (A0) edge node[midway,left] [right] {$m$} +(90:1);
        	\path [-] (A0) edge node[midway,left] [above] {$m$} +(0:1);
        \end{scope}
    \end{tikzpicture}
    \caption{An order $4$ tensor and a tensor train representation. }
    \label{TT:fig:hosvd}
\end{figure}
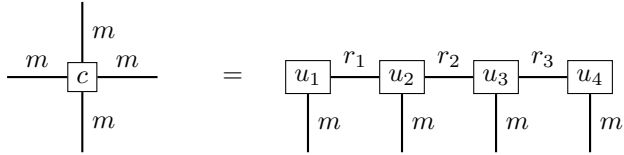

Tensor train representations of $c$ can now be defined as follows \cite{oseledets2011tensor}.
\begin{definition}[Tensor Train]\label{def:tensor_train}
    Let $c \in \mathbb R^{m \times \dots \times m}$.
    A factorization
    \begin{equation}
    \label{eq:TT rep}
        c = u_1 \circ u_2 \circ \dots \circ u_d,
    \end{equation}
    where $u_1 \in \mathbb R^{m \times r_1}$, $u_i \in \mathbb R^{r_{i-1} \times m \times r_i}$, $2 \leq i \leq d-1$, $u_d \in \mathbb R^{r_{d-1} \times m}$, is called \emph{tensor train representation} of $c$. 
    We say that $u_i$ are \emph{component tensors}. The tuple of the dimensions $(r_1, \dots, r_{d-1})$ is called the representation rank and is associated with the specific representation \eqref{eq:TT rep}.
    In contrast to that, the tensor train rank (TT-rank) of $c$ is defined as the minimal rank tuple $\mathbf r = (r_1, \dots, r_{d-1})$, such that there exists a TT representation of $c$ with representation rank equal to $\mathbf r$. Here,  minimality of the rank is defined in terms of the partial order relation on $\mathbb N^d$ given by
\[ \mathbf s \preceq \mathbf t \iff s_i \leq t_i \text{ for all } 1 \leq i \leq d,\]
for $\mathbf r = (r_1, \dots, r_d), \, \mathbf s = (s_1, \dots, s_d) \in \mathbb N^d$.
\end{definition}

It can be shown that every tensor has a TT-representation with minimal rank, implying that the TT-rank is well defined \cite{holtz2012manifolds}.
An efficient algorithm for computing a minimal TT-representation is given by the  Tensor-Train-Singular-Value-Decomposition (TT-SVD) \cite{oseledets2009breaking}.
Additionally, the set of tensor trains with fixed TT-rank forms a smooth manifold, and if we include lower ranks, an algebraic variety is formed \cite{landsberg2012tensors}.

The TT-representation of \eqref{eq:V c} is then given as
\begin{multline}
    \hat v(x) = \sum_{i_1}^m \cdots \sum_{i_d}^m \sum_{j_1}^{r_1} \cdots \sum_{j_{d-1}}^{r_{d-1}} u_1[i_1, j_1] u_2 [j_1, i_2, j_2] \cdots  \\
    \cdots u_d[j_{d-1}, i_d] \phi_{i_1}(x_1) \cdots \phi_{i_d}(x_d).
\end{multline}
Introducing the compact notation
\[ \phi: \mathbb R \to \mathbb R^{m}, \quad \phi(x) = [\phi_1(x), \dots, \phi_{m}(x) ], \]
the corresponding graphical TT-representation (with $d=4$ for definiteness) is then given as in Figure \ref{fig:valuefun1}. 
\begin{figure}[h!]
    \centering
    \begin{tikzpicture}
        \begin{scope}[every node/.style={draw,  fill=white}]
        \node (A1) at (0,0) {$u_1$}; 
        \node (A2) at (1.25,0) {$u_2$}; 
        \node (A3) at (2.5,0) {$u_3$}; 
        \node (A4) at (3.75,0) {$u_4$}; 
        
        \node (B1) at (0,-1) {$\phi(x_1)$}; 
        \node (B2) at (1.25,-1) {$\phi(x_2)$}; 
        \node (B3) at (2.5,-1) {$\phi(x_3)$}; 
        \node (B4) at (3.75,-1) {$\phi(x_4)$}; 
        \end{scope}
        \node (C0) at (-2,0) {$\hat v(x)$}; 
        \node (C1) at (-1,0) {$=$}; 
        \begin{scope}[every edge/.style={draw=black,thick}]
        	\path [-] (A1) edge node[midway,left,sloped] [above] {$r_1$} (A2);
        	\path [-] (A2) edge node[midway,left,sloped] [above] {$r_2$} (A3);
        	\path [-] (A3) edge node[midway,left,sloped] [above] {$r_3$} (A4);
        	\path [-] (A1) edge node[midway,left] [right] {$m$} (B1);
        	\path [-] (A2) edge node[midway,left] [right] {$m$} (B2);
        	\path [-] (A3) edge node[midway,left] [right] {$m$} (B3);
        	\path [-] (A4) edge node[midway,left] [right] {$m$} (B4);
        	
        \end{scope}
    \end{tikzpicture}
    \caption{Graphical representation of $v_n: \mathbb R^4 \to \mathbb R$.}
    \label{fig:valuefun1}
\end{figure}
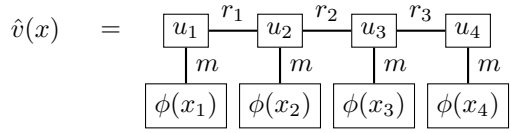

Within the TT format regression problems can be solved by using the \emph{alternating least squares} (ALS) algorithm \cite{ALS},
where the regression problem, suffering from the curse of dimensionality is reduced to a sequence of linear sub-problems, whose dimensionality is given by the size of a single component tensor. The routine is based on the observation, that fixing all component tensors but one reduces the multilinear ansatz to a linear problem, that can be solved by standard linear regression models, c.f. Figure \ref{fig:ALS}.
Note that we replace the integral in Figure \ref{fig:ALS} by Monte-Carlo quadrature, c.f. Section \ref{sect:vmc}.
Iteratively updating the component tensors produces a monotone sequence, that converges at least locally. 
\begin{figure}[h!]
    \centering
    \begin{tikzpicture}
        \begin{scope}[every node/.style={draw,  fill=white}]
        \node (A1) at (0,0) {$u_1$};

        \node (A3) at (2.5,0) {$u_3$}; 
        \node (A4) at (3.75,0) {$u_4$}; 
        
        \node (B1) at (0,-1) {$\phi(x_1)$}; 
        \node (B2) at (1.25,-1) {$\phi(x_2)$}; 
        \node (B3) at (2.5,-1) {$\phi(x_3)$}; 
        \node (B4) at (3.75,-1) {$\phi(x_4)$}; 
        \end{scope}
        \node (C1) at (-1.9,-0.5) {$\displaystyle \min_{\tilde u\in \mathbb R^{r_1\times m\times r_2}}\int_\Omega |$}; 
        \node[draw,circle] (A2) at (1.25,0) {$\tilde u$}; 
        \node (A12) at (0.8,0){};
        \node (A23) at (1.7,0){};
        \node (D2) at (1.25,-0.35){};
        \node (C2) at(6.3,-0.5){$-r(x_1,\dots,x_4)|^2\text{d}x_1\dots x_4$};
        \begin{scope}[every edge/.style={draw=black,thick}]
        	\path [-] (A1) edge   (A12);
        	\path [-] (A23) edge  (A3);
        	\path [-] (A3) edge   (A4);
        	\path [-] (A1) edge  (B1);
        	\path [-] (D2) edge (B2);
        	\path [-] (A3) edge (B3);
        	\path [-] (A4) edge (B4);
        	
        \end{scope}
    \end{tikzpicture}
    \caption{Graphical representation of ALS algorithm. In this linear sub-problem the second component tensor is being optimized, while component tensors $u_1,u_3,u_4$ are fixed (from previous iterations).}
    \label{fig:ALS}
\end{figure}
In our version of the ALS algorithm we update the component tensors from left to right.
We say that after every component tensor is updated once, one sweep is complete.
In order to circumvent overfitting, we further add a penalization term to the loss functional \eqref{eq:algo_regression}, such that the regression problem becomes
\begin{equation}\label{eq:regularizer}
    \hat v = \argmin_{v \in \mathcal M} \frac{1}{J} \sum_{j = 1}^J |v(x_j) - \tilde v(x_j) |^2 + \delta \| v \|_{H^2_{\text{mix}}(\Omega)},
\end{equation}
where $H^2_{\text{mix}}(\Omega)$ is the tensor product of one-dimensional $H^2$ Sobolev spaces \cite{sickel2009tensor}, assuming that $\Omega$ can be written as $\Omega = \bigotimes_{i = 1}^d [a, b]$, where $a < b$.
Choosing the one-dimensional ansatz functions to be orthonormal w.r.t. $H^2(a, b)$, this regularization term is realized via Parseval's identity, by penalizing the Frobenius norm of the component tensors.
As $H^2_{\text{mix}}(\Omega)$ is continuously embedded into $W^{1, \infty} (\Omega)$ \cite{sickel2009tensor} we are penalizing the $L^\infty(\Omega)$ norm of the gradient of $v$, which prevents overfitting.
Note that in order to reduce the impact of the regularizer we reduce $\delta$ within the iteration of the ALS.
After every completion of a sweep, we set $\delta$ to be the $10^{-3}$ times the residual of the loss functional \eqref{eq:regularizer}.
In our numerical tests we have noticed that this regularizer is integral to the success of our method.
\subsection{Complexity calculations for the evaluation and gradient of the function}
In the following we cover the evaluation of $v$ and the computation of the gradient of $v$ within the TT-format.
Note that for fast evaluation of the feedback law, efficient computation of the gradient is critical.
In the following complexity considerations we assume that $r_i = r$ for all $i$.

\textbf{Evaluating $v$ in the TT format.}
First, we consider the mapping $x \mapsto v(x)$ under the assumption that $v$ is in TT-format.
This operation can be denoted using the $\circ$ operation as
\begin{align}
    v(x_1, \dots, x_d) &= u_1 \circ \dots \circ u_d \circ \phi(x_d) \circ \dots \circ \phi(x_1). \label{eq:evaluate_tt}
\end{align}
We notice that in order to evaluate $v$ at a certain point $x \in \mathbb R^d$, we have to compute $\phi(x_i)$ for all $i$.
This corresponds to $m \cdot d$ evaluations of one-dimensional functions, where we assume that these evaluations are $\mathcal O(1)$ each.
The contraction is performed from left-to-right (or from right-to-left) and is of complexity $\mathcal O(d m r^2)$, because for every component tensor of dimension $r \times m \times r$ one vector of size $r$ and one vector of size $m$ has to be contracted, leaving a vector of size $r$.

\textbf{Evaluating $\nabla v$ in the TT format.}
In order to evaluate $\nabla v$ at a certain point $x \in \mathbb R^d$, we have to compute $\phi (x_i)$ and $\phi'(x_i)$ for all $i$, where $\phi'(x_i) \in \mathbb R^m$ is the derivative of the one-dimensional functions $\phi_i$ stacked into a vector.
Next we observe that because of the tensor structure of our basis, the partial derivative $\frac{\partial v}{\partial x_i}$ is given by
\begin{multline}
    \frac{\partial v}{\partial x_i}(x_1, \dots, x_d) =  \\
    u_1 \circ \dots \circ u_d \circ \phi(x_d) \circ \dots \circ \phi(x_{i +1}) \circ \phi'(x_{i}) \circ \phi(x_{i-1}) \circ \dots \circ \phi(x_1).
\end{multline}
Noticing the similarity of the above formulas with \eqref{eq:evaluate_tt} we can again estimate the complexity of computing a partial derivative to be similar to the previous case.
A naive implementation of the gradient then yields another factor $d$ in the complexity estimation, obtaining a total complexity of $\mathcal O(d^2 m r^2)$.
However, in the TT-format the naive implementation computes many redundant contractions, which means that we can save some complexity here.
As shown in Appendix \ref{sec:fast_grad} it is possible to reduce the complexity to $\mathcal O(d m r^2)$ via a recurrent scheme.

\section{Comparing both approaches and possible improvements}\label{sect:compare}
In this section we compare both approaches and give rise to possible improvements/adaptions of the algorithms.

\subsubsection*{Comparison}
We notice that both algorithms share the same backwards iteration to approximate $v^*(t_l, \cdot)$ for all $t_l$.
Thus, we have to compare how the local optimal control problems are solved.

Within the policy iteration approach the value function is approximated in an iterative scheme, where for every iteration step trajectories of length $\tau_k$ have to be computed and then a linear equation has to be solved.
In contrast to that, for the open-loop approach, an optimal control problem has to be solved for every sample point.
After that a single linear equation has to be solved.
Consequently, generating the samples for the open-loop approach is more expensive, provided that the policy iteration does not need too many iterations, while solving the linear equation is more expensive in the policy iteration case due to the iteration scheme. In \cite{fackeldey2020approximative} first ideas to the formal scaling w.r.t. the spatial dimension are described for the Policy Iteration combined with Least-Squares methods in the context of stochastic exit time problems. As indicated there within, the complexity of deterministic systems is in a similar but reduced fashion since the stochastic behaviour does not need to be resolved.

For both approaches, we draw samples $x_i \in \Omega$ and keep the same samples during the complete iteration.
Within the backwards iteration we have to choose initial data for both approaches.
In the policy iteration approach, an initial policy has to be chosen.
Due to the short time-horizon of the local optimization problems it is in most cases possible to choose the $0$ policy.
However, in many cases choosing the policy from the previous step is valid as well, i.e. setting $v_0(t, \cdot) = \hat v(t_{l+1}, \cdot)$, $t \in [t_l, t_{l+1})$ and obtaining $\alpha_0$ using the optimality condition from Theorem \ref{thm:optimality_condition}.
In our numerical tests we use the latter approach.

For the open-loop approach we do not need an initial policy.
Instead, we have to choose initial controls for the open-loop solver.
Here, it is again possible to choose the $0$ control.
However, using the data that was generated in the previous time step can increase the convergence rate tremendously.
More exactly, we denote by $u_{i,l}$ the initial guess, to emphasize its dependency on the sample $x_i$ and on the initial time $t_l$.
We use use the control that was computed in the previous time step, i.e. $u_{i,l} = u^*_{i, l+1}$, which is close to optimal if the difference between the final conditions is small.
After optimizing $u_{i,l}$ we denote the optimal control by $u^*(i, l)$.

For both approaches we notice that generating the samples can be parallelized perfectly.
In the case of the policy iteration approach we have to solve an ODE on a short time frame, whereas in the open-loop approach an optimal control problem has to be solved, which involves a gradient descent scheme where several forward and backward ODEs have to be solved.

\subsubsection*{Possible improvements}
The first improvement we propose is based on the error propagation from Theorem \ref{thm:error_propagation} and the Bellman equation.
We first observe that the Bellman equation does not only hold for initial time $t_l$ and final time $t_{l+1}$, but instead for every final time larger than $t_l$.
Assuming that we obtain the same error $\delta$ for the regression w.r.t. every end-point larger than $t_l$, larger time-horizon decreases the error propagation.
By setting the end-point to be $t_{l+2}$, the error bound from Theorem \ref{thm:error_propagation} is halved.
Setting the final time to be $T$ for every initial time $t_l$ prevents any error propagation.
Of course, solving the open-loop control problems with longer time-horizon becomes increasingly difficult.
Here, a balance between time-horizon and error propagation has to be found.

In a similar way it is possible to increase the time horizon for the policy iteration approach.
This is done by integrating along longer trajectories as indicated in Figure \ref{fig:horizon}
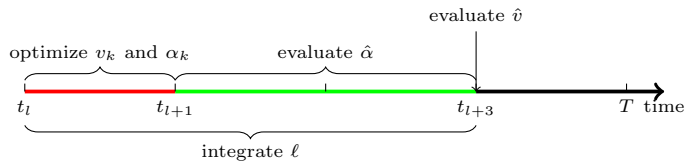
\begin{figure}[h]
    \centering
\begin{tikzpicture}[%
    every node/.style={
        font=\scriptsize,
        text height=1ex,
        text depth=.25ex,
    },
]
\draw[red, line width = 0.5mm] (0,0) -- (2,0);
\draw[green, line width = 0.5mm] (2,0) -- (6,0);
\draw[->, line width = 0.5mm] (6,0) -- (8.5,0);
\draw[->] (6,.8) -- (6,0);
\node at (6,1) {evaluate $\hat v$};

\foreach \x in {0,2,...,8}{
    \draw (\x cm,3pt) -- (\x cm,0pt);
}

\node[anchor=north] at (0,0) {$t_l$};
\node[anchor=north] at (2,0) {$t_{l+1}$};
\node[anchor=north] at (6,0) {$t_{l+3}$};
\node[anchor=north] at (8,0) {$T$};
\node[anchor=north] at (8.5,0) {time};

\draw[decorate,decoration={brace,amplitude=5pt}] (0,0.15) -- (2,0.15)
    node[anchor=south,midway,above=4pt] {optimize $v_k$ and $\alpha_k$};
\draw[decorate,decoration={brace,amplitude=5pt}] (2,0.15) -- (6,0.15)
    node[anchor=north,midway,above=4pt] {evaluate $\hat \alpha$};
\draw[decorate,decoration={brace,amplitude=5pt}] (6,-0.45) -- (0,-0.45)
    node[anchor=north,midway,below=4pt] {integrate $\ell$};
\end{tikzpicture}
\caption{Visualization of increased integration horizons. The red part is optimized via the policy iteration, the green part is only evaluated and the evaluate or black part is estimated by evaluating $\hat v$ at time $t_{l+3}$ . Note that $\hat \alpha$ and $\hat v$ are already computed and thus fixed.}\label{fig:horizon}
\end{figure}

The second proposed improvement can only be used by the open-loop approach.
The idea is using the optimal controls $u_{i,l}^*$ that we compute via the gradient descent method.
Due to Theorem \ref{thm:optimality_condition}, we know that 
\begin{equation}
    u^*_{i, l}(t_l) = - \frac 1 2 R^{-1} g(t, x_i)' \nabla v^*(t_l, x_i).
\end{equation}
This additional information can be incorporated to a modified regression problem, where we add a quadratic loss term containing the control
\begin{equation}\label{eq:regularizer2}
    \hat v_l = \argmin_{v \in \mathcal M} \frac{1}{L} \sum_{l = 1}^L |v(x_l) - \tilde v(x_l) |^2 + 
    \eta |u^*_{i, l}(t_l) + \frac 1 2 R^{-1} g(t, x_i)' \nabla v(x_i)|^2 +
    \delta \| v \|_{H^2_{\text{mix}}(\Omega)},
\end{equation}
where $\eta \geq 0$.
Adding the loss w.r.t. to the gradient of $v$  is closely related to the so-called \emph{physical informed neural network} (PINN) approach known in deep learning \cite{RAISSI2019686} and we also want to highlight that in \cite{azmi2020optimal} a similar adaption of the loss functional is done, where improved performance is reported.

Finally, another improvement is again motivated by MPC.
Instead of computing the feedback law by using the optimality condition in Theorem \ref{thm:optimality_condition} we can compute it by an MPC approach using the value function as final condition.
Depending on the length of the time horizon this can be done in real-time, yielding an optimal feedback law.
Due to the error propagation (backwards in time) it can be expected that the error in the value function is lower at later time steps.
Moreover, the gradient of the value function does not appear directly in the feedback law, but only indirect in the adjoint method within the MPC method.

A possible adaption of the open-loop algorithm is to decouple the value function and the controller.
In the open-loop approach it is possible to solve separate regression problems for approximating $\tilde v$ and $u^*$.
As the information on $u^*$ is a byproduct of the open-loop ansatz this does not increase the complexity of generating the data.
However, instead of one regression problem for the value function, two regression problems have to be solved with one being the value function and the other being the controller.

\section{Numerical Results}
We present results of numerical tests for different optimal control problems. For the implementation of the tensor networks we use the library \code{xerus} \cite{xerus}. 
The calculations were performed on a AMD Ryzen 5 PRO 3500U 8x 2.60GHz, 16 GB RAM Fedora 33 Linux distribution and the code is available on \url{https://github.com/lsallandt/finitehorizon_bellman}.
In every test we consider a cost functional of the form
\begin{equation}\label{eq:num_cost}
    \argmin_{u \in L^2((0, \infty); \mathbb R^m)} \mathcal{J}(x, u) = \int_{0}^{T} \| y(t) \|^2 + 0.1 \|u(t)\|^2\ + c \| y(T) \|^2 dt,
\end{equation}
where $c \geq 0$ and a PDE
\[ \dot y = f(y) + g(y) u, \quad y \in L^2(-1,1). \]
As the first step we discretize the PDE in space, such that we obtain a finite dimensional system of ODEs, which we also denote as
\[ \dot y = f(y) + g(y) u, \quad y \in \mathbb R^n. \]
For this discretization we use simple finite differences methods.
We implement our algorithms for the spatially discretized PDE.
As polynomial ansatz spaces we use the tensor product of one-dimensional $H^2$-orthogonal polynomials of degree smaller than $4$ and we use the loss functional with regularizer \eqref{eq:regularizer}.
We benchmark the controllers obtained by our optimization by comparing it to the optimal open-loop control over the whole time-horizon $[0, T]$.
We compute this optimal control by using the same gradient descent method used for the local optimal control problems.
In fact, we use a simple fixed step-size and for the update direction we use the current gradient and the gradient of the previous step.
In particular for Test 1 finding the optimal control over the whole time horizon without a good initial control is non-trivial, and computing the first gradient fails, due to the instability of the system. 
Thus, we use the control generated by our feedback controllers as initial controls.
This problem did not occur for the local optimal control problems due to their short time horizon.
We further compare our controllers to the linear quadratic regulator (LQR), a closed-loop controller that is obtained by linearizing the systems around $0$ and then solving the Riccati equation.
We denote this controller by $\alpha_{\text{LQR}}$.
In the numerical tests we discretize the time-dependency of the value function using $\tau = \tau_l = t_{l+1} - t_l = 0.01$.
The ODE is discretized using a different step-size, namely $0.001$ and the explicit Runge-Kutta $4$ method, which means that for every step in the value function, $10$ steps of the ODE are computed.
Here, we use linear interpolation of the value function to obtain the value function at the steps between our discretization points $t_l$.
We state that this uncoupling is integral for the numerical success of our method.
Setting $\tau = \tau_l = 0.001$ we obtain worse results.
This effect can be attributed to the error propagation from Theorem \ref{thm:error_propagation}.
\begin{remark}\label{rem:pictures}
In the following tests, we distinguish between the policy $\alpha$, the corresponding cost estimator $v$ and the real generated cost $\mathcal J(\cdot, \alpha(\cdot))$. For fixed $x$, we obtain $v(x)$ by simply evaluating $v$. Here, no trajectory has to be computed. We obtain $\mathcal J(x, \alpha(x))$ by numerically integrating along the trajectory with initial condition $x$. Note that $\mathcal J(x, \alpha(x))$ is basically the numerical approximation of the cost functional with respect to a feedback law, defined in \eqref{eq:feedback_cost}.
\end{remark}

\subsection{Test 1:  Diffusion with Unstable Reaction Term}
We consider a diffusion equation with unstable reaction term and Neumann boundary condition, c.f. \cite[Test 2]{pol_approx_kunisch}. Solve \eqref{eq:num_cost} with $c = 1$ and $T = 0.3$ for $y \in L^2(-1,1)$ subject to
\begin{align*}
\dot y &= \sigma \Delta y + y^3 + \chi_\omega u \\
y(0) &= x \\
\end{align*}
with Neumann boundary condition and $\chi_\omega$ is the characteristic function w.r.t. $\omega = [-0.4, 0.4] \subset [-1,1]$. We choose $\sigma = 1$ and use a finite differences grid with $d \in \mathbb N$ grid points to discretize the spatial domain. We denote by $A$ this finite difference discritization of the Laplace operator and by $G\in \{0,1\}^d$ the discritization of the characteristic function $\chi_\omega$. Then we obtain a system of $d$ ordinary differential equations 
\begin{align*}
    &\dot y = \sigma A y + y^3 + Gu\\
    & y(0)=x
\end{align*}
Using the step-size $h = \frac{1}{d+1}$ we get a finite dimensional approximation of the term $\| y(t) \|_H^2$ in the cost functional. For this test we choose a spatial dimension of $d=32$. As the underlying equation is non linear, our ansatz for the value function is the tensor product of polynomials up to degree $4$. The internal ranks chosen are
\[ [3, 4, 5, 5, 5, 6, 6, 6, 6, 7, 7, 7, 7, 7, 7, 7, 7, 7, 7, 6, 6, 6, 6, 6, 6, 6, 5, 5, 5, 4, 3]. \]
We solve the HJB equation in 32 dimensions on the set $[-2,2]^d$. While the full ansatz space has dimension $5^{32}$, the TT has $5395$ degrees of freedom.
For the calculations we use $32370$ uniformly distributed Monte-Carlo samples.
\begin{remark}
We stress that the number of Monte-Carlo samples is extremely small when comparing it to the dimension of the ambient space \newline $32370 \ll 5^{32} \approx 10^{22}$ and only when comparing it to the degrees of freedom in the TT representation the numbers become comparable, $32370 = 6 \cdot 5395$.
This further indicates, that the curse of dimensionality is broken by our ansatz. Additional studies, where the minimal number of samples is compared to the dimensions of the underlying systems have to be done.
\end{remark}

We first test the feedback controllers for certain initial values, visualized in Figure \ref{fig:s_test_certain}. 
For both controllers, significant improvement in cost is noticeable, with the greatest being approximately $38\%$ of the cost saved compared to the LQR controller.
Moreover, we see that the computed feedback laws generate close to optimal costs for the tested initial values.
\begin{figure}[h]
\centering
\begin{subfigure}[c]{0.49\textwidth}
\includegraphics[width=1\textwidth]{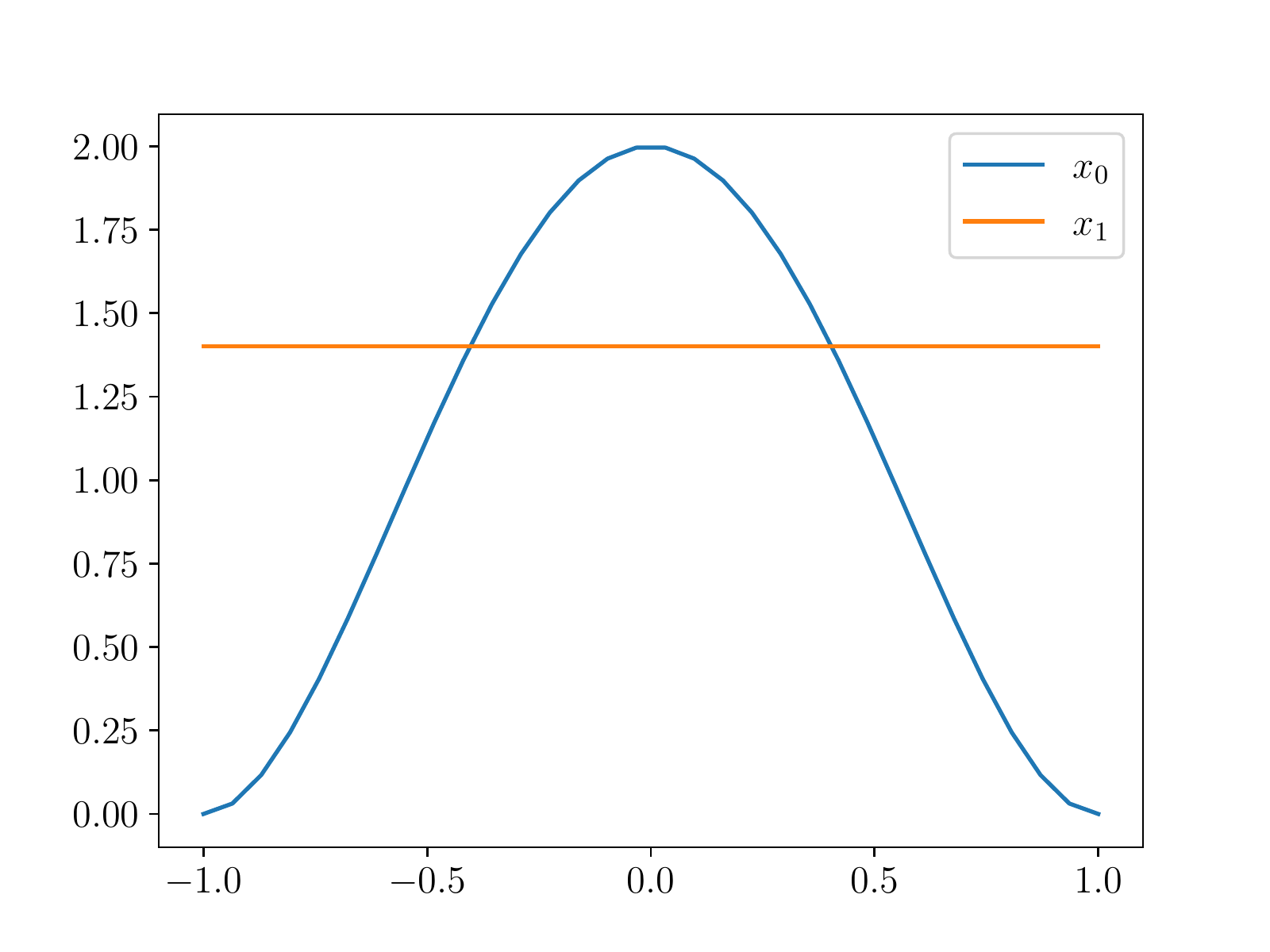}
\subcaption{Initial values $x_0$ and $x_1$.}
\end{subfigure}
\begin{subfigure}[c]{0.49\textwidth}
\includegraphics[width=1\textwidth]{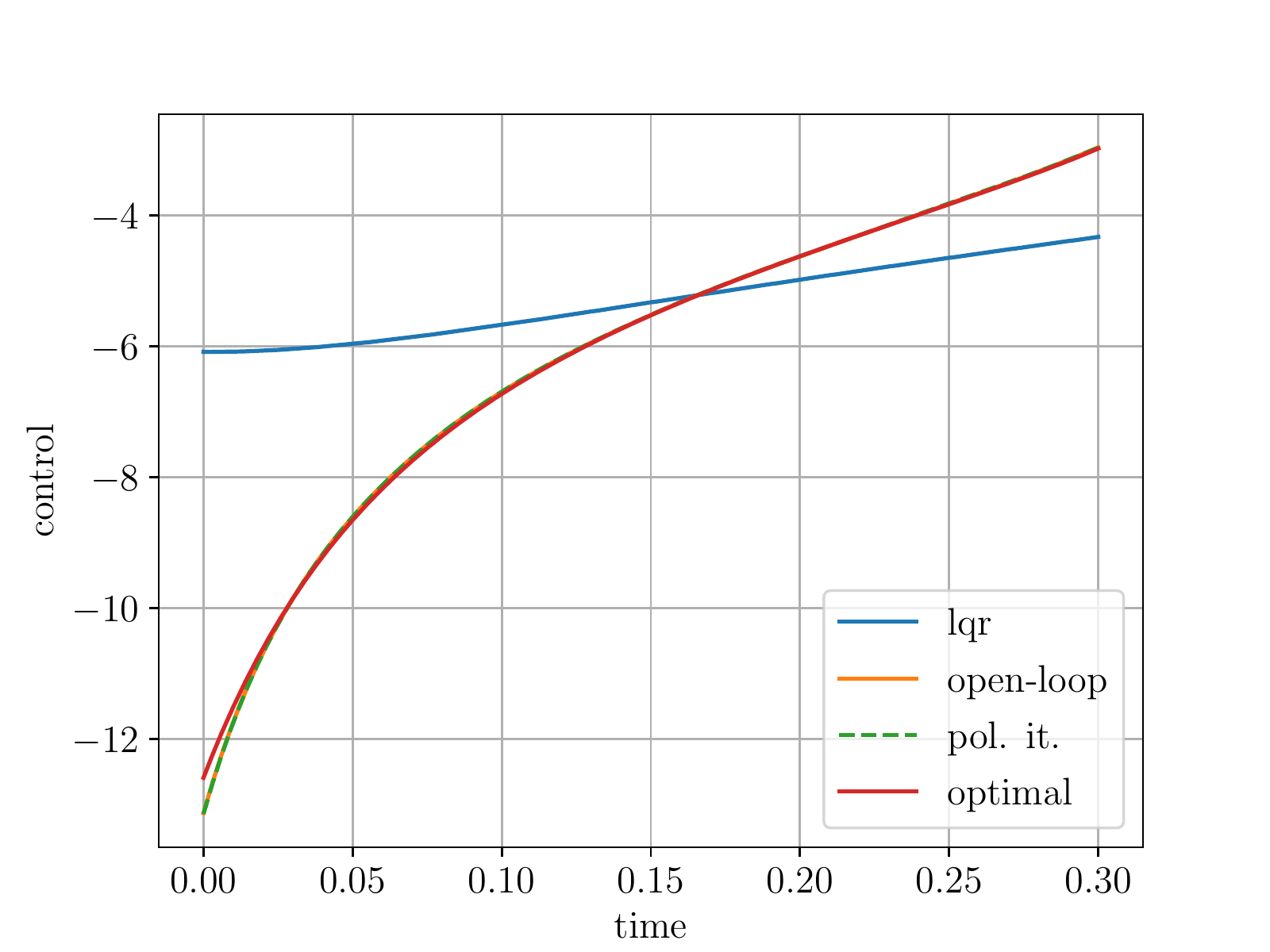}
\subcaption{Generated controls, initial value $x_0$.}
\end{subfigure}
\hfill
\begin{subfigure}[c]{0.49\textwidth}
\includegraphics[width=1\textwidth]{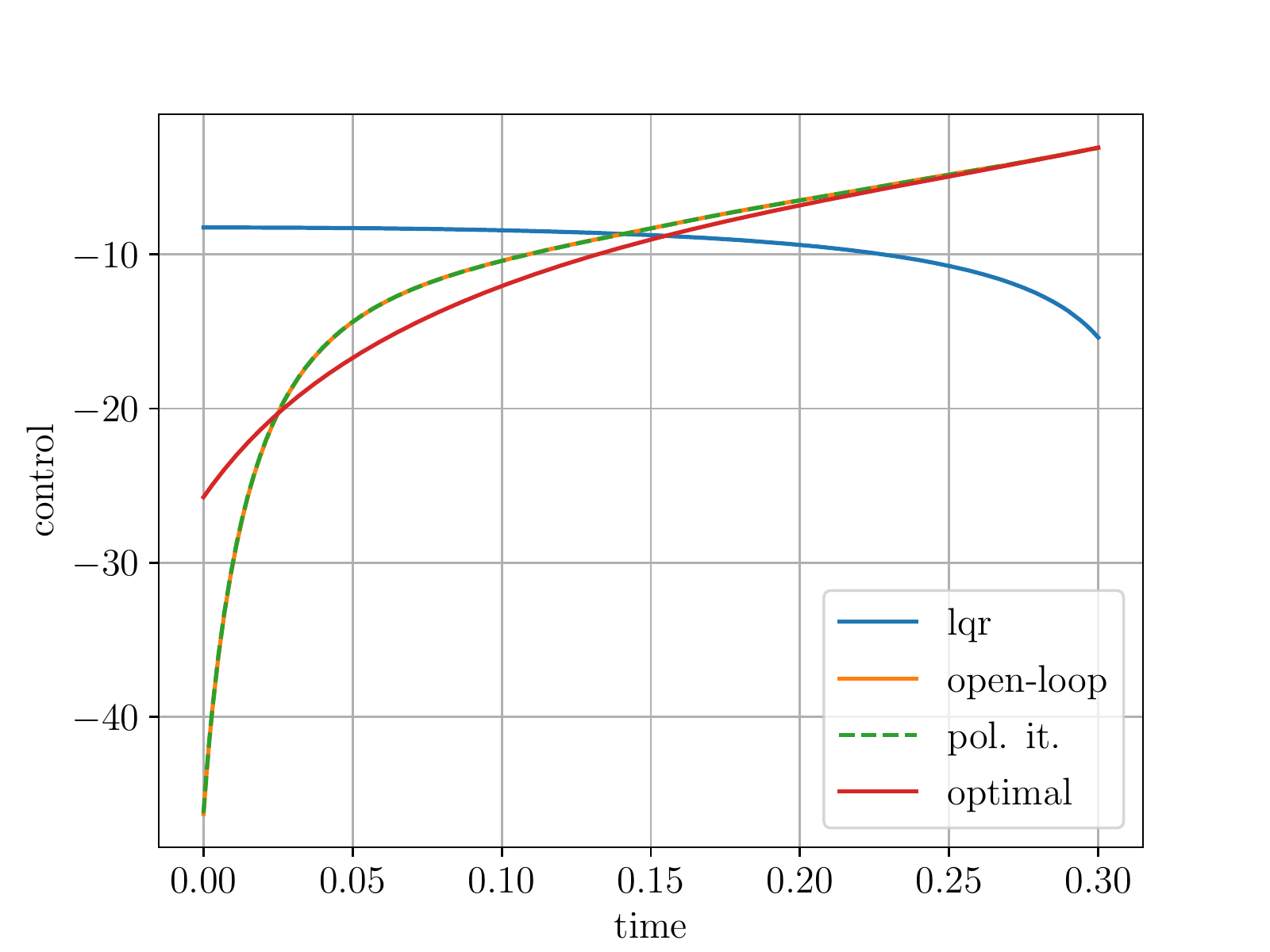}
\subcaption{Generated controls, initial value $x_1$.}
\end{subfigure}
\begin{subfigure}[c]{0.49\textwidth}
\includegraphics[width=1\textwidth]{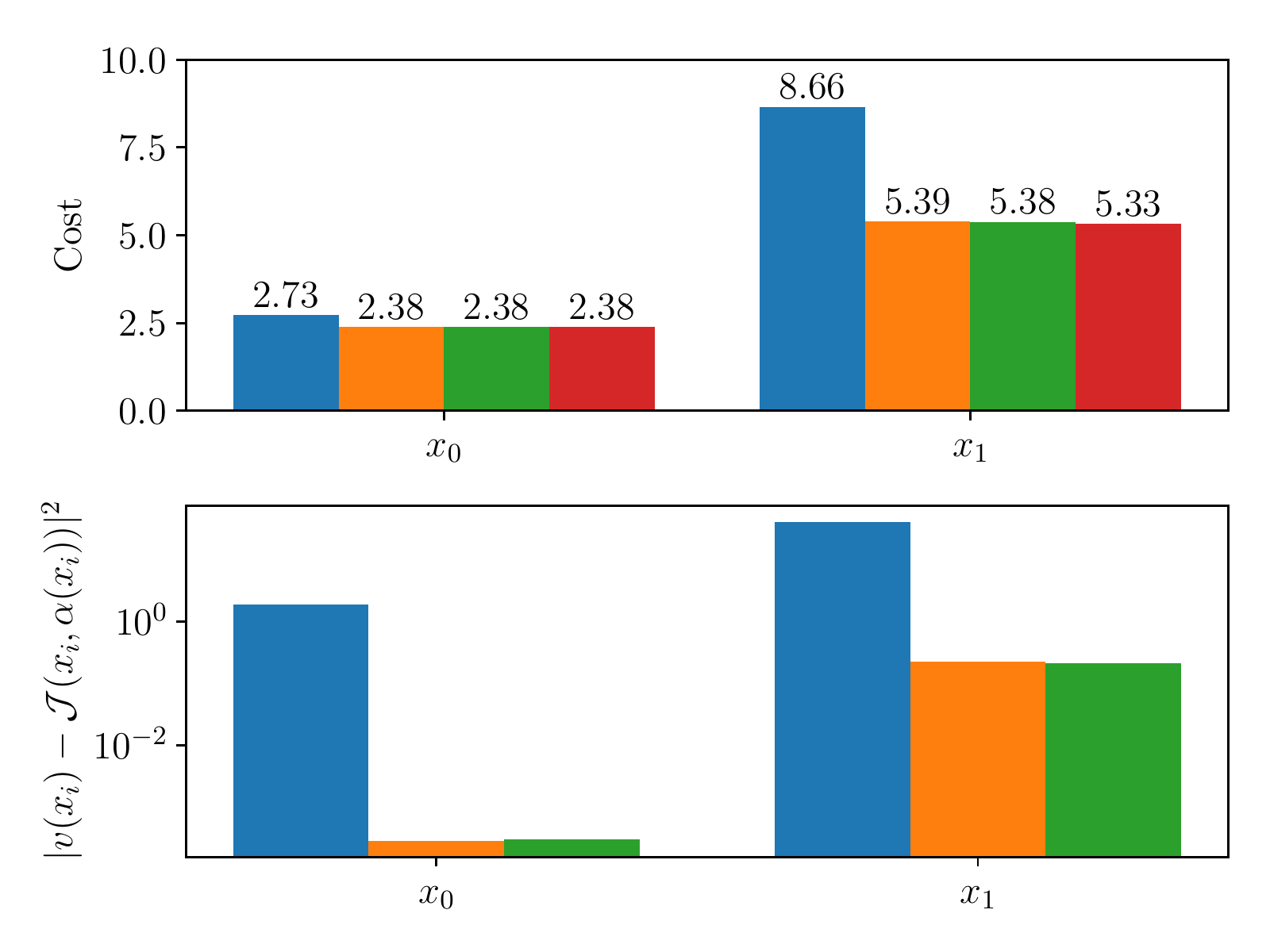}
\subcaption{Generated cost and accuracy of the value function. Blue ist the LQR controller, orange is the open-loop ansatz, green is the policy iteration ansatz, red is the optimal control.}
\end{subfigure}
\caption{The generated controls and cost for different initial values.}\label{fig:s_test_certain}
\end{figure}
We further investigate the performance of the controller by choosing initial values of the type $[x, x, \dots, x]$, where $x 
\in [0, 2]$.
In Figure \ref{fig:s_xxx_initial_values} we see that for small $x$ every controller is close to optimal.
For $x$ larger than $1$ the LQR controller performs significantly worse than the other controllers.
By increasing the initial values beyond $1.5$, the LQR controller fails to stabilize the system, while the our controllers still stabilize the system.
However, for such values the computed feedback laws  differs evidently  from the 
optimal control. 
For such large initial values the present setting was to coarse to achieve better accuracy.
Note that for computing the optimal control for such extreme initial values a good initial guess is needed.
In particular, due to the blow-up, it is not possible to use the control generated by the LQR controller as initial guess.
We were only able to find the optimal control by using the controls generated by our feedback controllers as initial guesses.
\begin{figure}[h]
\centering
\includegraphics[width=.5\textwidth]{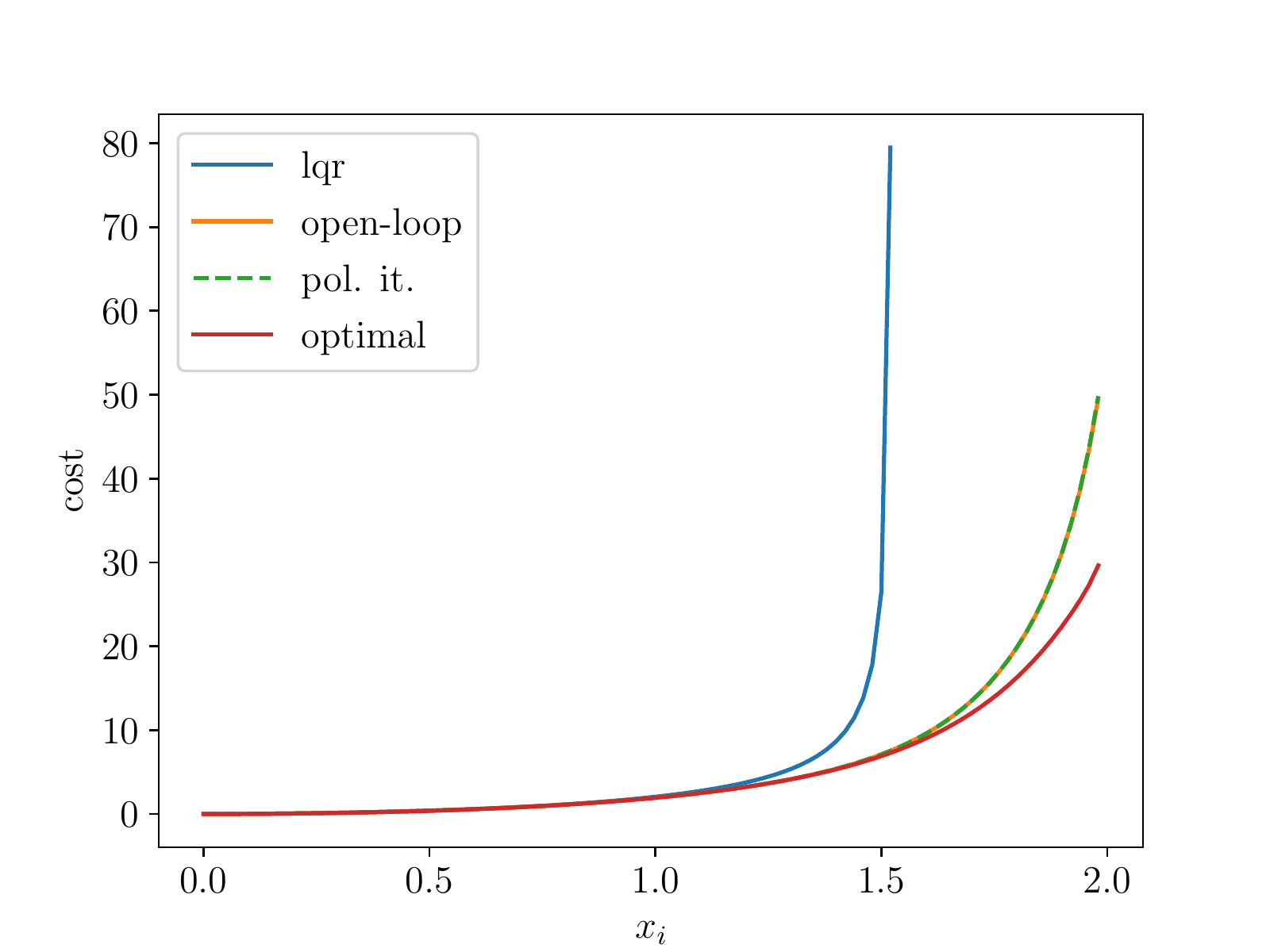}
\caption{Cost of initial values of the type $[x, x, \dots, x]$.}\label{fig:s_xxx_initial_values}
\end{figure}
Next we test the feedback law for random initial values. Note that because of the diffusion, equally distributed samples and normally distributed samples yield low cost on average and in this case no improvements of the cost is to be expected. Thus, we use a special distribution of initial values that we specify now. For every initial value we choose an equally distributed integer between $2$ and $20$. This number is the degree of a random polynomial. Next we choose a polynomial with normal distributed coefficients of the degree we chose. We further modify the polynomial in the following way $p(x) := \tilde p(x) (x-1)(x+1)$, such that we have $p(-1) = p(1) = 0$. Finally, we rescale $p$ such that its maximum in $[-1,1]$ is $1.9$. In order to have an idea how these initial values look, we plotted $10$ initial values in Figure \ref{fig:s_x0_random} and report the results in Table \ref{tab:s_random}.
We report that for these initial values the LQR controller failed to obtain cost smaller than $100$ in $65$ out of $1000$ initial values, while our controllers were not only stabilizing for every initial value, but also close to optimal.
On the set of initial values that the LQR controller did not fail we see an average improvement of $32 \%$ while being close to optimal.
Finally, we report that on the whole set of initial values, which means that these where the LQR controller failed are included, the average difference to the optimal control was $1\% $.
We note that out of all $1000$ samples, the largest relative difference between the optimal control and our feedback controllers is $\approx 34\%$, which was also observed in Figure \ref{fig:s_xxx_initial_values} for values close to the boundary of our integration area.
\begin{figure}[h]
\centering
\includegraphics[width=.5\textwidth]{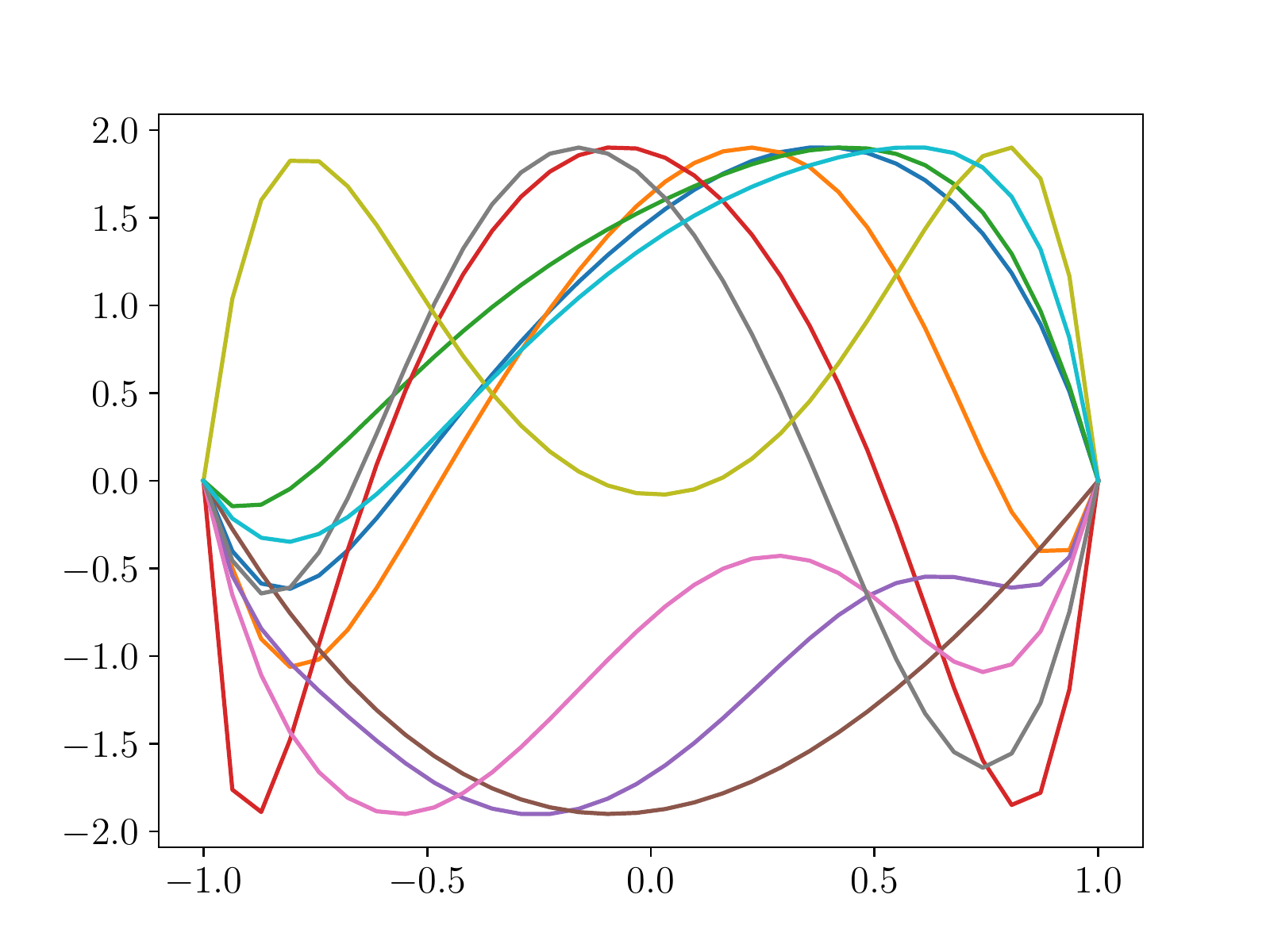}
\caption{Examples of  samples drawn from the polynomial distribution.}\label{fig:s_x0_random}
\end{figure}
\begin{table}[h]
\begin{center}
\begin{tabular}{ c | c c c c}
controller & \% cost $< 100$ & avg. cost & max. rel. diff. to opt. & avg. Bellman error \\
\hline
 LQR & $93.5$ & $4.453$ & $\text{nan}$ & $39.67$\\ 
 open-loop & $100$ & $2.61$ & $0.3419$ & $0.048$\\  
 pol. it. & $100$ & $2.61$ & $0.3381$  & $0.047$ \\  
 optimal & $100$ & $2.60$ &  $0$ &   
\end{tabular}
\end{center}
\caption{Performance of the different controllers for $1000$ samples drawn from the polynomial distribution. The averaged values are only taken from the subset of initial values that the LQR succeeded in stabilizing.}\label{tab:s_random}
\end{table}
\FloatBarrier
\subsection{Test 2: Allen-Kahn equation}\label{subsec:burgers}
As underlying equation we use a one-dimensional Allen-Kahn equation similar to \cite[Test 3]{pol_approx_kunisch}. Solve \eqref{eq:num_cost} for $y \in L^2(-1,1)$ subject to
\begin{align*}
\dot y &= \sigma \Delta y + y - y^3 + \chi_\omega u \\
y(0) &= x
\end{align*}
with Neumann boundary condition. Here, we use the same discretization as in Section \ref{subsec:burgers}. The constants are the same except for $\sigma = 0.2,\ \omega = [-0.5, 0.2]$. 
Again, an ansatz of polynomials up to degree $4$ is used.
We choose the same ranks as in the last test
\[ [3, 4, 5, 5, 5, 6, 6, 6, 6, 7, 7, 7, 7, 7, 7, 7, 7, 7, 7, 6, 6, 6, 6, 6, 6, 6, 5, 5, 5, 4, 3] \]
and solve the HJB on $[-2,2]^{32}$.

From Figure \ref{fig:a_test_certain} we deduce that for certain initial values, significant improvement of cost is possible for both controllers with the greatest improvement being $25\%$ of the cost.
Again, the open-loop and the policy iteration approach yield similar performance.
We again notice that for these initial values our calculated value functions are more accurate than the predictions from the LQR-based value function.
\begin{figure}[htb]
\centering
\begin{subfigure}[c]{0.49\textwidth}
\includegraphics[width=1\textwidth]{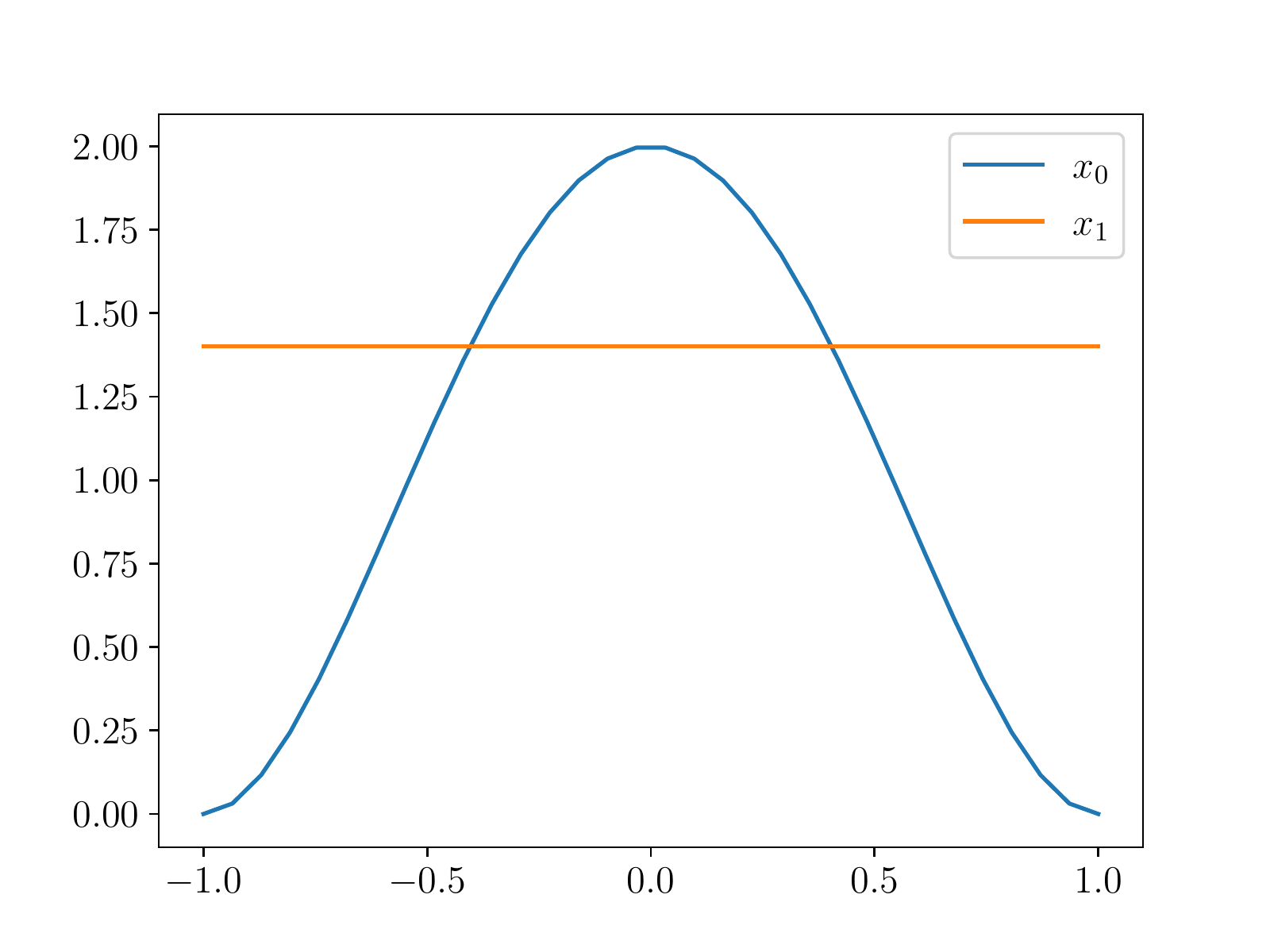}
\subcaption{Initial values $x_0$ and $x_1$.}
\end{subfigure}
\begin{subfigure}[c]{0.49\textwidth}
\includegraphics[width=1\textwidth]{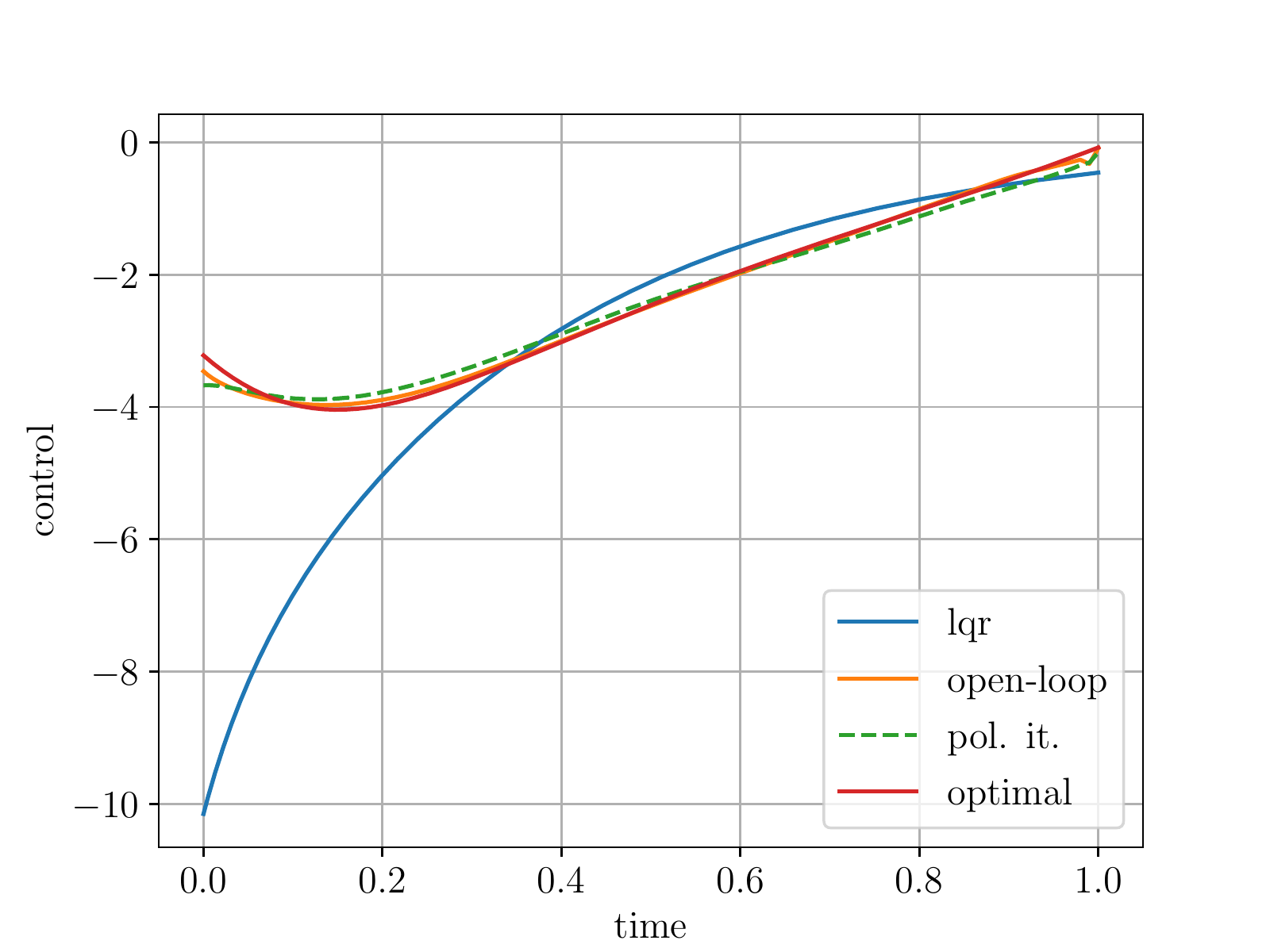}
\subcaption{Generated controls, initial value $x_0$.}
\end{subfigure}
\hfill
\begin{subfigure}[c]{0.49\textwidth}
\includegraphics[width=1\textwidth]{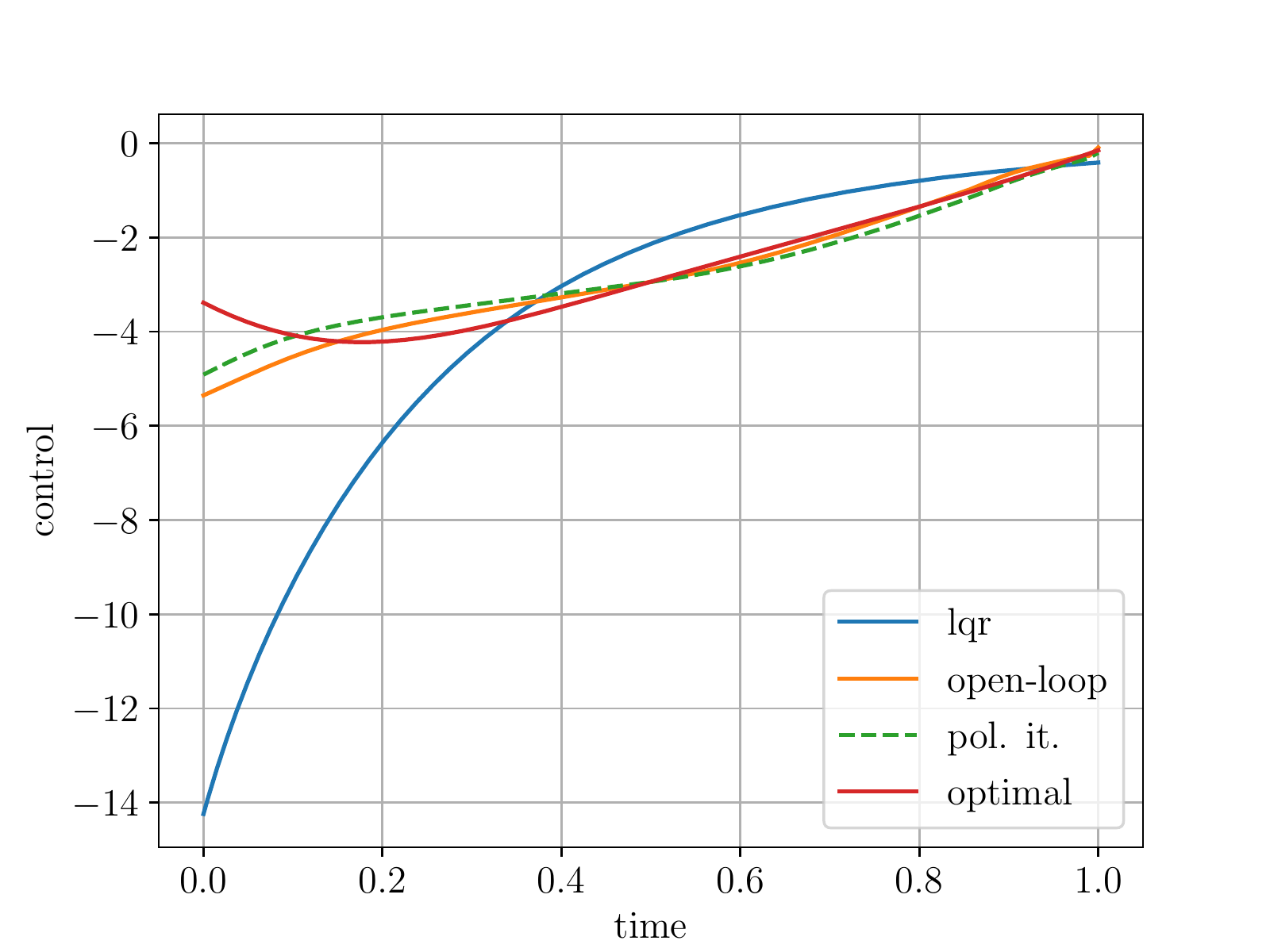}
\subcaption{Generated controls, initial value $x_1$.}
\end{subfigure}
\begin{subfigure}[c]{0.49\textwidth}
\includegraphics[width=1\textwidth]{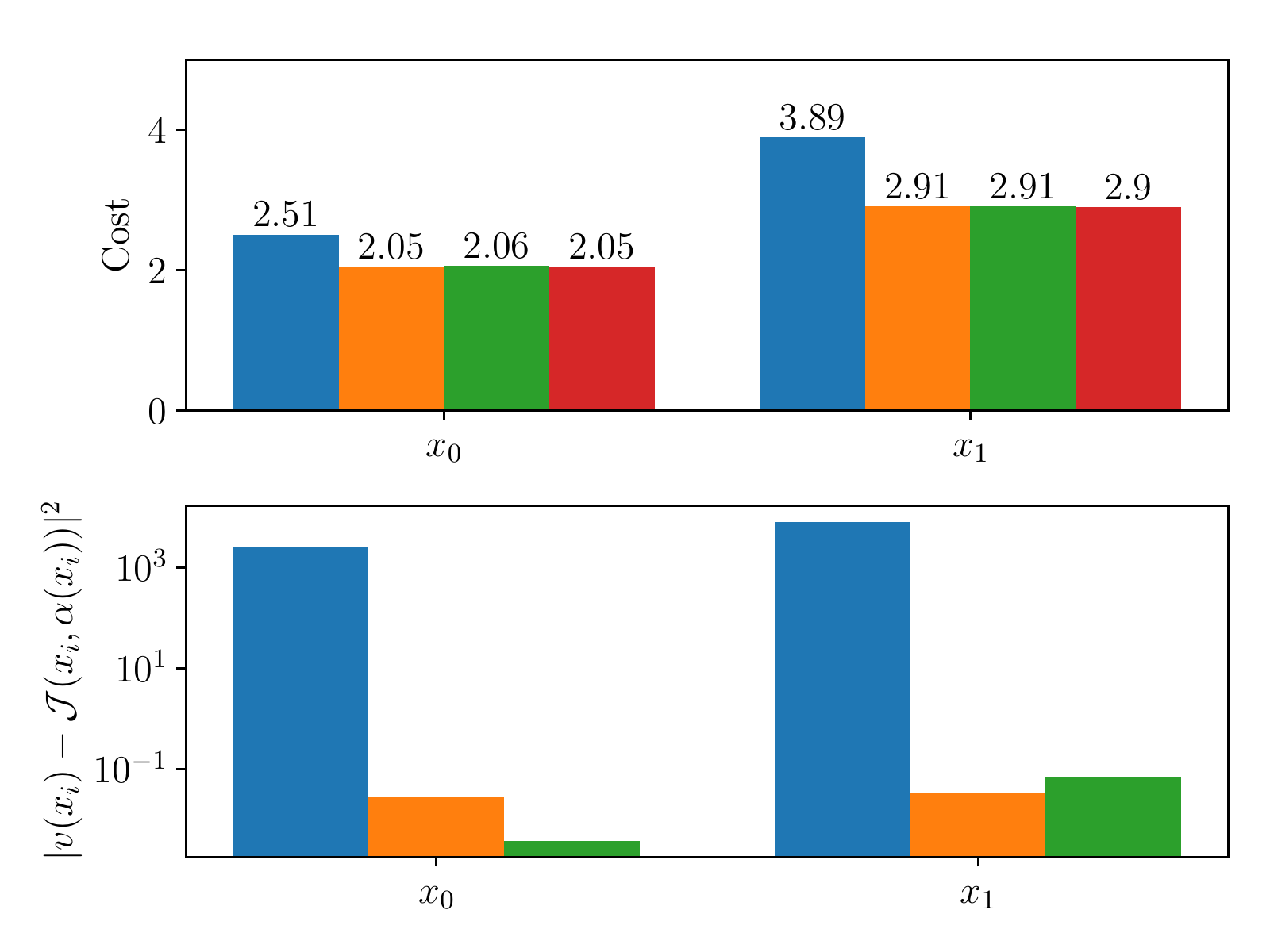}
\subcaption{Generated cost and accuracy of the value function. Blue ist the LQR controller, orange is the open-loop ansatz, green is the policy iteration ansatz, red is the optimal control.}
\end{subfigure}
\caption{The generated controls and cost for different initial values.}\label{fig:a_test_certain}
\end{figure}

We again further investigate the performance of the controller by choosing initial values of the type $[x, x, \dots, x]$, where $x 
\in [0, 2]$.
In Figure \ref{fig:a_xxx_initial_values} we see that for small $x$ every controller is close to optimal.
For $x$ larger than $0.5$ the LQR controller yields significantly worse performance than the other controllers.
However, in this case the LQR controller stabilizes the system for every initial value.
\begin{figure}[htpb]
\centering
\includegraphics[width=.5\textwidth]{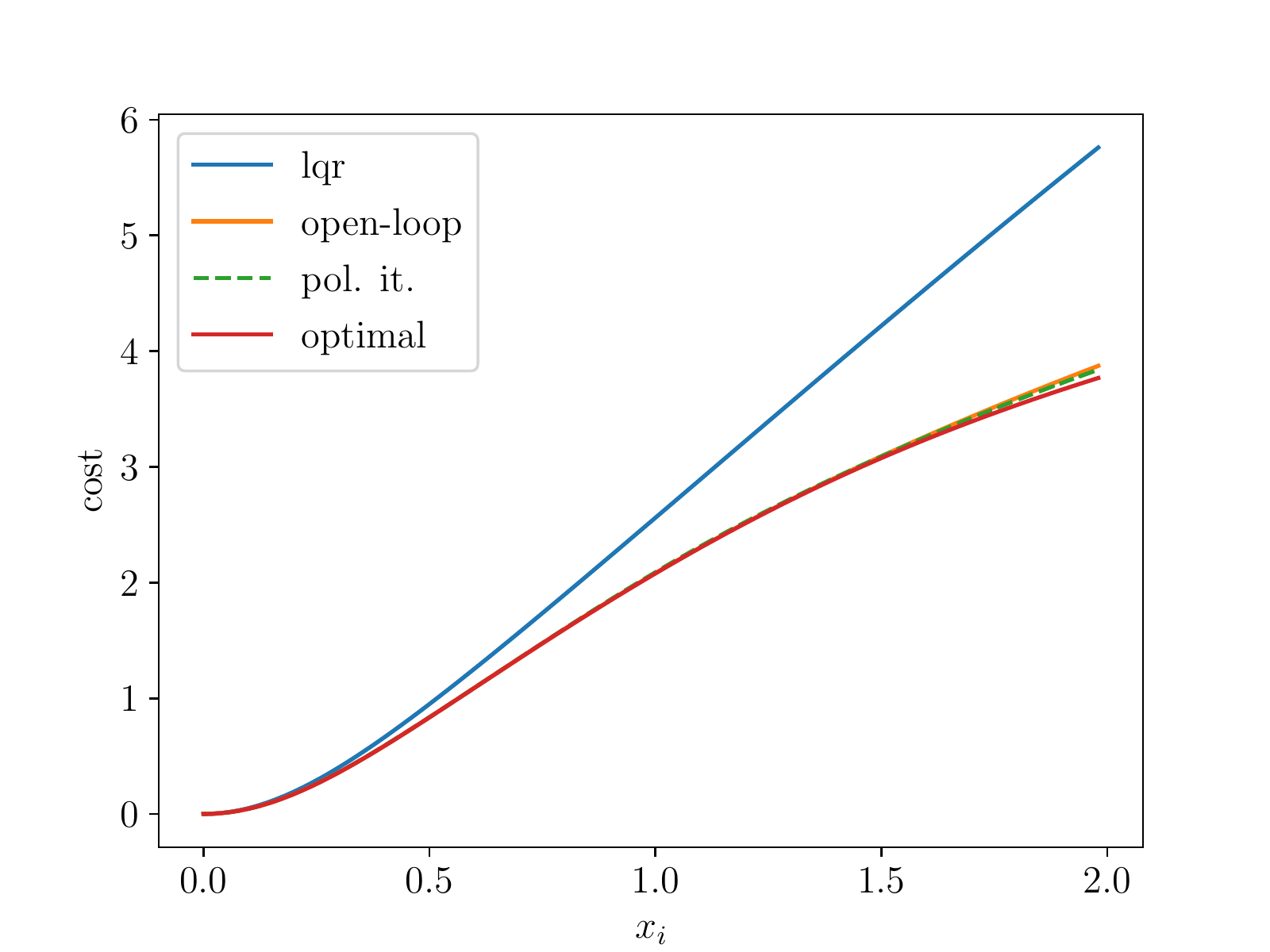}
\caption{Cost of initial values of the type $[x, x, \dots, x]$.}\label{fig:a_xxx_initial_values}
\end{figure}

Next we again test random initial values using the same setup as in the previous test.
In Table \ref{tab:a_random} we compare the performance of the controllers for $1000$ random initial values drawn from the polynomial distribution.
Again, the open-loop and the policy iteration approach yield close to optimal performance, while the LQR generates $22 \%$ more cost.
Moreover, the the open-loop and the policy iteration approaches predict the generated cost accurately.
\begin{table}[htpb]
\centering
\begin{center}
\begin{tabular}{ c | c c c c}
controller & \% cost $< 100$ & avg. cost &  max. rel. diff. to opt. & avg. Bellman error \\
\hline
 LQR & $100$ & $2.434$ & $0.4234$ &  $24172$\\ 
 open-loop & $100$ & $1.985$ & $0.0146$ & $0.032 67$\\  
 pol. it. & $100$ & $1.988$  & $0.0152$ & $0.030593$ \\  
 optimal & $100$ & $1.981$  & $0$  &  
\end{tabular}
\end{center}
\caption{Performance of the different controllers for $1000$ samples drawn from the polynomial distribution. The averaged values are only taken from the subset of initial values that the LQR succeeded in stabilizing.}\label{tab:a_random}
\end{table}

\section*{Conclusion}
We have compared two methods for finding optimal controllers for finite horizon optimal control problems.
In numerical tests we have observed similar performance for both methods.
In contrast to the linear quadratic regulator we have obtained close to optimal costs for many initial states.
Moreover, in none of the tests blow-ups occurred for our controllers.
By encoding an approximation of the value function in a low-rank tensor model we have obtained a low-fidelity representation that allows fast evaluation of the feedback law.
If even higher accuracy is needed, it is possible to take the control generated by our low-fidelity model as initial guess for further improvements via an open-loop optimization, where the control generated by the LQR controller fails.
%

\section*{Acknowledgements}
Leon Sallandt and Mathias Oster acknowledge support from the Research Training Group ”Differential Equation- and Data-driven Models in Life Sciences and Fluid Dynamics: An Interdisciplinary Research Training Group (DAEDALUS)” (GRK 2433) funded by the German Research Foundation (DFG).
\printbibliography

 \appendix
 
\section{Proof}
\begin{proof}[proof of Lemma \ref{lem:v_cont_wrt_final}]
Towards a contradiction assume that there exists an initial value $x$ such that $| v_1(x) - v_2(x) | > \delta$. W.l.o.g. assume that $v_1(x) - v_2(x) > \delta$.
Due to the definition of $v_2$, for every $\varepsilon>0$ there is a control $u_2$ such that
$$\int_{t_0}^{t_1} \ell(y_2,u_2) dt + c_2(y_2(t_1))-\varepsilon\leq v_2(x),$$
where $y_2$ is the trajectory corresponding to $u_2$.
Plugging $u_2$ into the first cost functional yields
\begin{align*}
    v_1(x) < \int_{t_0}^{t_1} \ell(y, u_2) dt + c_1(y(t_1)) := \tilde v_1(x).
\end{align*}
Now it follows that
\begin{align*}
\delta < v_1(x) - v_2(x) \geq \tilde v_1(x) - v_2(x) &\leq \int_{t_0}^{t_1} \ell(y_2, u_2) dt  + c_1(y_2(t_1))   - (\int_{t_0}^{t_1} \ell(y_2, u_2) dt + c_2(y_2(t_1))  )+\varepsilon \\
&= c_1(y_2(t_1))  - c_2(y_2(t_1))+\varepsilon. 
\end{align*} 
Since this holds for every $\varepsilon>0$, this contradicts $\|c_1 - c_2 \|_\infty< \delta$.\\\\

\end{proof}

\section{Fast gradient evaluation}\label{sec:fast_grad}
We now show how to compute the gradient of a function in TT-representation in complexity $\mathcal O(d m r^2)$.
For this, we define the contractions
\begin{align}
    \psi_i^+(x_{i+1}, \dots, x_d) &= u_{i+1} \circ \dots \circ u_d \circ \phi(x_d) \dots \circ \phi(x_{i+1}) \\
    \psi_i^-(x_1, \dots, x_{i-1}) &= u_{1} \circ \dots \circ u_{i-1} \circ \phi(x_{i-1}) \dots \circ \phi(x_1).
\end{align}
Note that $\psi_i^+$ is the contraction of every component tensor with larger index than $i$, while $\psi_i^-$ is the contraction of every component tensor with smaller index than $i$.
We observe that $\psi_i^+(x_{i+1} \dots, x_d) \in \mathbb R^{r}$, $\psi_i^-(x_1 \dots, x_{i-1}) \in \mathbb R^{r}$ and that 
\begin{equation}\label{TT:eq:partial_derivative_v_recursive}
    \frac{\partial v}{\partial x_i}(x_1, \dots, x_d) = \psi_i^- \circ u_i \circ \psi_i^+ \circ \phi'(x_i).
\end{equation}
Finally, the recursive properties 
\begin{equation}\label{TT:eq:gradient_psi_recursive}
   \psi_i^+ = u_{i+1} \circ \psi_{i+1}^+ \circ \phi(x_{i+1}), \quad \psi_i^- = \phi(x_{i-1}) \circ (\psi_{i-1}^- \circ u_{i-1}) 
\end{equation}
yields Algorithm \ref{TT:algo:gradient}.

\begin{algorithm}[h]
\SetAlgoLined
\caption{Computing the gradient of a function $v$ in TT-format.}\label{TT:algo:gradient}
\SetKwInOut{Input}{input}\SetKwInOut{Output}{output}
\SetKwInOut{Output}{output}\SetKwInOut{Output}{output}
\Input{A function $v$ in TT-format, component tensor $u_1, \dots, u_d$, one-dimensional basis functions $\phi_1, \dots, \phi_n$, point $x = (x_1, \dots, x_d) \in \mathbb R^d$.}
\Output{The gradient $\nabla v(x)$}
\For{$\mu = 1, \dots, d-1$}{
    Calculate $\psi_i^-$ using the recursive formula \eqref{TT:eq:gradient_psi_recursive}.
}
\For{$\mu = d, \dots, 1$}{
    Calculate $\psi_i^+$ using the recursive formula \eqref{TT:eq:gradient_psi_recursive}.
}
\For{$\mu = d, \dots, 1$}{
    Calculate $\nabla v(x) [i] = \frac{\partial v}{\partial x_i}(x)$ by using \eqref{TT:eq:partial_derivative_v_recursive}.
}
\end{algorithm}
Note that the computational complexity of \eqref{TT:eq:partial_derivative_v_recursive} and \eqref{TT:eq:gradient_psi_recursive} is $\mathcal O(m r^2)$.
Thus, Algorithm \ref{TT:algo:gradient} has computational complexity of $\mathcal O(d m r^2)$.
This complexity compares to the complexity of the naive implementation $\mathcal O(d^2 m r^2)$.

\end{document}